\documentclass[11pt]{amsart}
\usepackage{amsmath,amssymb,amsfonts,latexsym}
\usepackage{mathrsfs}
\usepackage[dvips]{geometry}
\usepackage{array}
\usepackage{epsf}
\usepackage{epsfig}
\newtheorem*{lemma}{Lemma}
\newtheorem*{prop}{Proposition}
\newtheorem*{thm}{Theorem}
\newtheorem*{cor}{Corollary}
\newcommand{\iso}{\overset{\sim}{\rightarrow}}

\newcommand{\nc}{\newcommand}
\nc{\ad}{\operatorname{ad}} \nc{\tr}{\operatorname{tr}}
\nc{\tp}{\operatorname{top}} \nc{\rank}{\operatorname{rank}}
\nc{\corank}{\operatorname{corank}}
\nc{\codim}{\operatorname{codim}} \nc{\sdim}{\operatorname{sdim}}
\nc{\mult}{\operatorname{mult}} \nc{\Index}{\operatorname{Index}}
\nc{\spn}{\operatorname{span}} \nc{\Sym}{\operatorname{Sym}}
\nc{\sym}{\operatorname{sym}} \nc{\id}{\operatorname{id}}
\nc{\Id}{\operatorname{Id}} \nc{\Ree}{\operatorname{Re}}

\nc{\htt}{\operatorname{ht}}
\nc{\Ker}{\operatorname{Ker}}
\nc{\rker}{\operatorname{rKer}}
\nc{\im}{\operatorname{Im}}
\nc{\osp}{\mathfrak{osp}}
\nc{\sgn}{\operatorname{sgn}}
\nc{\F}{\operatorname{F}}
\nc{\Mod}{\operatorname{Mod}}
\nc{\Mat}{\operatorname{Mat}}
\nc{\Soc}{\operatorname{Soc}}
\nc{\Inj}{\operatorname{Inj}}
\nc{\Hom}{\operatorname{Hom}}
\nc{\End}{\operatorname{End}}
 \nc{\supp}{\operatorname{supp}}
\nc{\Card}{\operatorname{Card}}
\nc{\Ann}{\operatorname{Ann}}
\nc{\Ind}{\operatorname{Ind}}
\nc{\Coind}{\operatorname{Coind}}
\nc{\wt}{\operatorname{wt}}
\nc{\ch}{\operatorname{ch}}
\nc{\Stab}{\operatorname{Stab}}
\nc{\Sch}{{\mathcalS}\mbox{\emch}}
\nc{\Irr}{\operatorname{Irr}}
 \nc{\Spec}{\operatorname{Spec}}
\nc{\Prim}{\operatorname{Prim}}
\nc{\Aut}{\operatorname{Aut}}
\nc{\Ext}{\operatorname{Ext}}

\nc{\Fract}{\operatorname{Fract}}
 \nc{\gr}{\operatorname{gr}}
\nc{\deff}{\operatorname{def}}
\nc{\HC}{\operatorname{HC}}

\begin{document}

\title[Slices]{Slices for biparabolics of Index 1$^1$ }

\author[Anthony Joseph and Florence Millet]
{}

\date{\today}
\maketitle

$$\begin {matrix}&$\textsc{Anthony Joseph}$\\&$Donald Frey Professional
Chair$\\&$Department of Mathematics$
\\&$Weizmann Institute of Science$\\&$2 Herzl Street$\\&$Rehovot, 76100,
Israel$\\&$anthony.joseph@weizmann.ac.il$ \\\end{matrix}
\quad \quad\quad \begin {matrix}&$\textsc{Florence
Fauquant-Millet}$\\&$Universit\'e de Lyon, F-42023 Saint-Etienne,
France$\\&$Laboratoire de Math\'ematiques de l'Universit\'e de
Saint-Etienne$\\&$Facult\'e des Sciences et Techniques$\\&$23, rue
du Docteur Paul Michelon$
\\&$F-42023
Saint-Etienne C\'edex 02 France$
\\&$florence.millet@univ-st-etienne.fr$
\\\end{matrix}$$

 \footnotetext[1]{Work supported in part by
Israel Science Foundation Grant, no. 710724.}

%
%
%
%
%
%

\

\

\medskip


\

 \textbf{Abstract}

 Let $\mathfrak a$ be an algebraic Lie subalgebra of a simple Lie
 algebra $\mathfrak g$ with index $\mathfrak a \leq \rank \mathfrak g$.
 Let $Y(\mathfrak a)$ denote the algebra of
 $\mathfrak a$ invariant polynomial functions on $\mathfrak a^*$.
 An algebraic slice for $\mathfrak a$ is an affine subspace $\eta
 +V$ with $\eta \in \mathfrak a^*$ and $V \subset \mathfrak a^*$ a
 subspace of dimension index $\mathfrak a$ such that restriction of
 function induces an isomorphism of $Y(\mathfrak a)$ onto the
 algebra $R[\eta+V]$ of regular functions on $\eta+V$.

 Slices have been obtained in a number of cases through the
 construction of an adapted pair $(h,\eta)$ in which $h \in
 \mathfrak a$ is ad-semisimple, $\eta$ is a regular element of
 $\mathfrak a$ which is an eigenvector for $h$ of eigenvalue minus one and
 $V$ is an $h$ stable complement to $(\ad \mathfrak a)\eta$ in
 $\mathfrak a^*$.  The classical case is for $\mathfrak g$ semisimple
 \cite {K1}, \cite {K2}.  Yet rather recently many
 other cases have been provided. For example if $\mathfrak g$ is of type $A$ and
 $\mathfrak a$ is a ``truncated biparabolic" \cite {J7} or a
 centralizer \cite {J8}. In some of these cases (particular
 when the biparabolic is a Borel subalgebra) it was found \cite
 {J8}, \cite {J9}, that $\eta$ could be taken to be the restriction of a
 regular nilpotent element in $\mathfrak g$.  Moreover this
 calculation suggested \cite {J8} how to construct slices outside type $A$
 when no adapted pair exists.

 This article makes a first step in taking these ideas further.
 Specifically let $\mathfrak a$ be a truncated biparabolic of index one.
 (This only arises if $\mathfrak g$ is of type $A$ and $\mathfrak a$ is
 the derived algebra of a parabolic subalgebra whose Levi factor has just two blocks whose sizes are
 coprime.)   In this case it is
 shown that the second member of an adapted pair $(h,\eta)$
 for $\mathfrak a$ is the restriction of a particularly carefully
 chosen regular nilpotent element of $\mathfrak g$.

 A by-product of the present analysis is the construction of an
 invariant associated to a pair of coprime integers.

\section{Introduction}

Unless mentioned to the contrary the base field $\mathbb K$ is
assumed algebraically closed of characteristic zero.

\subsection{Invariants}\label{1.1}

Let $\mathfrak a$ be a finite dimensional Lie algebra,
$S(\mathfrak a)$ its symmetric algebra and $K(\mathfrak a)$ the
field of fractions of $S(\mathfrak a)$.  If $A$ is algebra in
which $\mathfrak a$ acts by derivations, set $A^\mathfrak a = \{a
\in A|xa =0, \forall x \in \mathfrak a \}$. It is a subalgebra of
$A$.

Given $\xi \in \mathfrak a^*$, set $\mathfrak a^\xi= \{a \in
\mathfrak a| a\xi =0\}$, that is the stabilizer of $\xi$ under
co-adjoint action.  It is a Lie subalgebra of $\mathfrak a$.

Define index $\mathfrak a:= \min_{\xi \in \mathfrak a^*}\dim
\mathfrak a^\xi$.  Set $\mathfrak a^*_{reg} = \{\xi \in \mathfrak
a^*|\dim \mathfrak a^\xi = \text {index}\ \mathfrak a\}$, called
the set of regular elements of $\mathfrak a^*$.

\smallskip

A problem of Dixmier \cite [Problem 4]{D} suggests that
$C(\mathfrak a):=K(\mathfrak a)^\mathfrak a$, is always a pure
transcendental extension of $\mathbb K$.

One may further ask under what conditions is $Y(\mathfrak
a):=S(\mathfrak a)^\mathfrak a$ a polynomial algebra.

\subsection{Slices}\label{1.2}

In \cite [Sect. 7] {J9} we focused some attention on refinements
of these questions.  Here it is convenient to assume that
$S(\mathfrak a)$ admits no proper semi-invariants.  In this case
$C(\mathfrak a)$ is just the field of fractions of $Y(\mathfrak
a)$.  Moreover under this hypothesis, Ooms and Van den Bergh \cite
[Prop. 4.1]{OV} have shown that the growth rate (that is
Gelfand-Kirillov dimension) of $Y(\mathfrak a)$ takes its maximum
possible value, namely index $\mathfrak a$.

Under the above hypothesis define a rational slice to be an affine
translate $\eta +V \subset \mathfrak a^*$ of a vector subspace of
$V$ of $\mathfrak a^*$ such that the restriction of functions
gives on injection $\theta$ of $Y(\mathfrak a)$ into the algebra
of regular functions $R[\eta+V]$ on $\eta +V$ and induces an
isomorphism of fields of fractions. Observe that $R[\eta+V]$
identifies with $S(V^*)$ and then comparison of transcendence
degrees implies that $\dim V = \text {index} \ \mathfrak a$.  We
call $\eta$ the base point of the slice $\eta+V$.

\smallskip

We suggested that a rational slice always exists \cite [7.11]{J9}.

\smallskip

Define an algebraic slice to be a rational slice for which
$\theta$ is an isomorphism.  Obviously this implies that
$Y(\mathfrak a)$ is a polynomial algebra; but we found an example
(\cite [11.4, Example 2]{J9}) for which the converse is false.

In view of this counter-example it would seem appropriate to
suggest that if $Y(\mathfrak a)$ is polynomial then there exists
an affine subspace $\eta +V \subset \mathfrak a^*$  such that
restriction of functions gives an embedding $Y(\mathfrak a)
\hookrightarrow R[\eta + V]\iso S(V^*)$ whose image takes the form
$S(V^*)^G$ for some finite (pseudo-reflection) group $G$ acting
linearly on $V^*$.

Finally we remark that the notions of a rational or algebraic
slice were given (\cite [Sect. 7]{J9} natural geometric
interpretations in the case when $A$ is a connected algebraic
group with Lie algebra $\mathfrak a$, that is when $\mathfrak a$
is algebraic. In particular $A(\eta + V)$ must be dense \cite
[7.9]{J9} (but not necessarily open \cite [11.4, Example 3] {J9})
in $\mathfrak a^*$. Thus $\eta +V$ must meet most regular orbits
(defined as those of codimension equal to index $\mathfrak a$).
However even in the case of an algebraic slice not every regular
orbit need pass \cite [8.12(ii)] {J4} through $\eta +V$, nor need
every orbit meeting $\eta +V$ be regular \cite [11.4, Example
3]{J9}.  In particular the base point $\eta$ need not be regular.

\subsection{Adapted Pairs}\label{1.3}

An adapted pair $(h,\eta)$ for a finite dimensional Lie algebra
consists of a regular element $\eta \in \mathfrak a^*$ and an
element $h \in \mathfrak a$ such $h\eta = -\eta$ with respect to
co-adjoint action.  Such pairs are rather hard to find, it being
particularly difficult to check regularity.

Assume that $\mathfrak a$ is an algebraic Lie algebra.  Then in
the above we may use Jordan decomposition to show that the
(adjoint) action of $h$ on $\mathfrak a$ can be taken to be
reductive without loss of generality. As $\mathfrak a^\eta$ is $h$
stable we may define $\{m_i\}_{i=1}^{ \text {index} \ \mathfrak
a}$, to be the set of eigenvalues (counted with multiplicities) of
$-h$ acting on $\mathfrak a^\eta$. These can be rather arbitrary
and may depend on the choice of the adapted pair \cite [8.3]{JS}.
However suppose that $S(\mathfrak a)$ admits no proper
semi-invariants and that $Y(\mathfrak a)$ is polynomial. Then by
\cite [Cor. 2.3]{JS} the degrees of the homogeneous generators of
$Y(\mathfrak a)$ are the $m_i+1:i=1,2,\ldots,\text {index}\
\mathfrak a$, and moreover $\eta +V$ is an algebraic slice for any
$h$ stable complement $V$ to $\mathfrak a \eta$ in $\mathfrak
a^*$. (This is actually proved under a slightly weaker hypothesis
which allows $\mathfrak a$ to be the centralizer $\mathfrak g^x$
in a semisimple Lie algebra $\mathfrak g$.)  Thus the $\{m_i\}$
generalize the so-called ``exponents" defined classically for
$\mathfrak a$ semisimple.

In the above situation every element of $\eta + V$ is regular (see
for example \cite [7.8]{J9}) by a standard deformation argument.

\subsection{The Nilpotent Cone}\label{1.4}

One would like to have a systematic way to construct algebraic
slices.   In this we make the rather bold suggestion below.  It
should be regarded more as a signpost rather than a serious
conjecture.

Suppose that $\mathfrak a$ is an algebraic subalgebra of a
semisimple Lie algebra $\mathfrak g$. Let $G$ be the adjoint group
of $\mathfrak g$ and $A$ the unique closed subgroup whose Lie
algebra is $\mathfrak a$.

Assume that index $\mathfrak a \leq \rank \mathfrak g$. This is
the case if $\mathfrak a$ is a biparabolic subalgebra \cite {FJ2},
\cite {J3} or a centralizer (via the now known truth of the
Elashvili conjecture \cite {CM}, \cite {G}, \cite {Y1}).

Further assume that $\mathfrak g$ admits a Chevalley
antiautomorphism $\kappa$ such that $\mathfrak a$ and
$\kappa(\mathfrak a)$ are non-degenerately paired through the
Killing form $K$ on $\mathfrak g$.  This is clearly the case for
(truncated) biparabolics.  It is well-known for a centralizer -
see \cite [4.5,4.6] {JS} for details and references.  It is less
clear that this condition is really necessary.  We use it mainly
for convenience.

Under the above hypothesis we may and will identify $\mathfrak
a^*$ with the subalgebra $\kappa(\mathfrak a)$ of $\mathfrak g$.

Let $\mathfrak k$ be the kernel of the restriction map $\mathfrak
g^* \rightarrow \mathfrak a^*$. Identifying $\mathfrak g^*$ with
$\mathfrak g$ through the Killing form we may view $\mathfrak k$
as a subspace of $\mathfrak g$.

Obviously $\mathfrak k$ is $A$ stable.  In view of the above
identifications we may write $A(\xi +\mathfrak k)=A\xi + \mathfrak
k$, for all $\xi \in \mathfrak a^*$.  In particular if $\xi \in
\mathfrak a^*_{reg}$, then codim $A(\xi+\mathfrak k) \leq \rank
\mathfrak g$.

Let $\mathscr N(\mathfrak g)$ denote the cone of ad-nilpotent
elements of $\mathfrak g$.  As is well-known, codim $\mathscr
N(\mathfrak g)=\rank \mathfrak g$.  Moreover $\mathscr N(\mathfrak
g)$ is irreducible and admits only finitely many $G$ orbits.  In
particular $\mathscr N(\mathfrak g)_{reg}$ consists of a dense
open orbit.

Now let $(h,\eta)$ be an adapted pair for $\mathfrak a$.  The
relation $h\eta = -\eta$, forces $\eta \in \mathscr N(\mathfrak
g)$.  However it is almost never the case that  $\eta \in \mathscr
N(\mathfrak g)_{reg}$. (For example, in \cite [Sect. 10] {J7} a
detailed study of the nilpotent orbit to which $\eta$ belongs was
made in the case of an adapted pair $(h,\eta)$ of a truncated
biparabolic in type $A$.)

In view of the above codimensionality estimates we propose the

\smallskip

\textbf{Suggestion}.  \textit{Suppose $(h,\eta)$ is an adapted
pair for $\mathfrak a$. Then $\eta+\mathfrak k \cap \mathscr
N(\mathfrak g)_{reg}$ is non-empty.}

\smallskip

\textbf{Remarks}.  This just means that there is some pre-image of
$\eta \in \mathfrak g^*$ lying in $\mathscr N(\mathfrak g)_{reg}$.
It could be proved by showing that the codimension of $\mathscr
N(\mathfrak g) \cap A(\eta +\mathfrak k)$ in $\mathfrak g^*$
equals $\rank \mathfrak g$, though this is likely to be rather
difficult if even true. It suggests that one should construct
$\eta$ as the restriction of some element of $\mathscr N(\mathfrak
g)_{reg}$. The fact that $\eta$ itself does not belong to
$\mathscr N(\mathfrak g)_{reg}$ is just a consequence of having
made a particular choice of its pre-image in $\mathfrak g$.
However it is this choice which allows one to guess $\eta$, itself
a rather hard task as explained in \cite [1.3]{J7}. A main point
that lies behind our suggestion is that there may be a more
``canonical" choice which leads to a slice making sense for a
biparabolic or centralizer of an arbitrary semisimple Lie algebra.
For the moment it is not too clear if this can be divined.
Hopefully the present article will provide a clue. Here we should
also stress that adapted pairs are far from unique even up to the
obvious conjugation \cite [1.4]{J7}.  Our suggestion is also
partly motivated by Question (5) of \cite [Sect. 11]{J7}.

\subsection{}\label{1.5}

Let us recall some cases in which the above suggestion has a
positive answer.

Choose a Cartan subalgebra $\mathfrak h$ of $\mathfrak g$ and let
$\Delta \subset \mathfrak h^*$ denote the set of non-zero roots of
$\mathfrak g$ with respect to $\mathfrak h$.  For each $\alpha \in
\Delta$ let $x_\alpha$ be a non-zero vector in $\mathfrak g$ of
weight $\alpha$.

let $\pi \subset \Delta$ be a choice of simple roots and set
$\Delta^+=\Delta\cap \mathbb N \pi$, $\mathfrak n = \sum_{\alpha
\in \Delta^+}\mathbb Kx_\alpha$ and $\mathfrak b = \mathfrak h +
\mathfrak n$, which is a Borel subalgebra.  Let $N,H,B$ be the
corresponding closed subgroups of $G$.

First assume that $\mathfrak a$ is a centralizer, that is of the
form $\mathfrak g^x$.  (Here we can assume $x$ nilpotent without
loss of generality and we shall always do this.)

Suppose $\xi$ is a regular element of $(\mathfrak g^x)^*$, for
example coming from the second factor in an adapted pair. Then
under the identifications made in \ref {1.4}, it follows from
\cite [Lemma 2.2]{J8} (which was inspired by the proof of the
Vinberg inequality) and the truth of the Elashvili conjecture that
$x+t\xi$ is a regular element of $\mathfrak g^*$ for all $t$
belonging to a cofinite subset $\Omega \subset \mathbb K$. However
in general $x+t\xi$ will not be nilpotent. Rather for $\xi$ in
general position $x+t\xi:t \in \Omega$ will be semisimple \cite
[5.5]{J8}.

Conversely if $\mathfrak g$ is of type $A$, then we may take $x$
in Jordan form (defined by an ordered partition $\textbf{x}$ of
$n$) and then ``complete" it to a standard regular nilpotent
element. More precisely up to conjugation we can write $\sum
_{\alpha \in \pi'} x_\alpha$, with $\pi' \subset \pi$
corresponding to $\textbf{x}$. Set $y'=\sum _{\alpha \in \pi
\setminus\pi'} x_\alpha$. Then $x+y'=\sum_{\alpha \in
\pi}x_\alpha$, which is the standard presentation of a regular
nilpotent element.

Now consider $y'$ as an element of $(\mathfrak g^x)^*$ through the
Killing form $K$.  (With respect to what we said in \ref {1.4} we
can arrange for $x, y:=\kappa(x)$ to generate a Jacobson-Morosov
s-triple containing $x$. In this $\mathfrak g^y= \kappa(\mathfrak
g^x)$, contains $y'$ and is non-degenerately paired to $\mathfrak
g^x$ through $K$, so identifies with $(\mathfrak g^x)^*$.) A basic
result proved in \cite [Thm. 4.7]{J8} is that there exists $h' \in
\mathfrak h\cap \mathfrak g^x$ making $(h',y')$ an adapted pair
for $\mathfrak g^x$. Now clearly $K(x,\mathfrak g^x)=0$ and so $x
\in \mathfrak k$.  We conclude that for this particular adapted
pair the suggestion of \ref {1.4} has a positive answer.

Notice that this construction makes sense for nilpotent orbits
generated by a subset of the simple root vectors (called
Bala-Carter orbits or orbits of Cartan type) for any simple Lie
algebra. However $y'$ obtained in this fashion is seldom regular
in $(\mathfrak g^x)^*$ and in fact regularity requires a very
careful choice of $\pi'$ . An interesting case is when card $\pi'
=1$, say $\pi' =\{\alpha\}$. Notice that if $\alpha$ is a long
root, then $\mathfrak g^{x_\alpha}$ is conjugate to the
centralizer of the highest root vector which is also a standard
truncated parabolic subalgebra. Now forgetting type $A_{2n}$,
which is just the case when the Coxeter number is odd, a simple
root system contains a distinguished long root defined in terms of
its Dynkin diagram. This is the central root in type $A_{2n+1}$,
the root with three neighbours in types $D,E$ and the unique long
root with a short root neighbour in types $B,C,F,G$. (For a
further interpretation relating this construction to the highest
root, see \cite [2.14]{J9}.)

If one chooses the (long) simple root $\alpha$ as above, then in
all cases (except $E_8$) the element $y'$ as defined above can be
completed to an adapted pair \cite [Sect. 6]{J8}; but this
generally fails if one takes $\alpha$ to be an arbitrary long
simple root. This again verifies our suggestion for that
particular pair, showing in addition that the question is rather
delicate. In type $E_8$, the element $y'$ is not regular \cite
[6.14]{J8} in $(\mathfrak g^x)^*$ and it is not known if the
latter algebra admits an adapted pair.   After Yakimova \cite {Y2}
the invariant algebra $Y(\mathfrak g^x)$ is not polynomial.

Our suggestion was also found to hold for some adapted pairs for
the (truncated) Borel subalgebra in type $A$.  In this case the
specification of a Borel subalgebra implies a choice of a set
$\pi$ of simple roots and it was found that the regular nilpotent
element in the conclusion of the suggestion was obtained from
$\sum_{\alpha \in \pi}x_{-\alpha}$ through conjugation by a rather
carefully chosen element of the Weyl group $W:=N_G(H)/H$.  It
turned out that this element of $W$ made sense for all simple Lie
algebras and through its use we were able to construct \cite [Thm.
9.4]{J9} an algebraic slice for a truncated Borel in all types
except $C,B_{2n},F_4$ even though an adapted pair does not exist
(outside type $A$). Here the base point $\eta$ was not regular but
still satisfied the conclusion of Suggestion \ref {1.4}. Obviously
we should like to take these last observations further.

We remark that the index of a (truncated) parabolic (resp.
biparabolic) was calculated in \cite {FJ2} (resp. \cite {J3}) and
that in most cases (all cases for types $A,C$) the invariant
algebra was shown to be polynomial (\cite {FJ1}, \cite {J3}).  In
type $A$ an adapted pair was constructed for all truncated
biparabolics \cite {J7}.  For a centralizer $\mathfrak g^x$ of a
simple Lie algebra $\mathfrak g$ the invariant algebra was shown
\cite {PPY} to be polynomial in many cases (all cases in types
$A,C$), whilst in type $A$, or for a long root vector outside type
$E_8$, the above construction of an adapted pair (which has the
additional property of being ``compatible") allows one to prove
very easily \cite [Thm. 3.5]{J8} the polynomiality of $Y(\mathfrak
g^x)$.

\subsection{}\label{1.6}

The purpose of the present article is to verify our suggestion for
(truncated) biparabolics of index one. As noted in \cite
[2.2,2.3]{J5}, these are described as the derived algebras of
maximal parabolic subalgebras in type $A$ for which the Levi
factor consists of two blocks of coprime sizes $p,q$.  In this we
shall take $p<q$ with the smaller block in the top left hand
corner. The parabolic is assumed to have Levi factor having these
two blocks and with nilradical $\mathfrak m$ being the
\textit{lower} left hand corner block and thus is spanned by root
vectors in which the ``non-compact" simple root, namely $\alpha_p$
in the Bourbaki notation \cite [Planche I] {B}, occurs with
coefficient $-1$ in every root of $\mathfrak m$. The truncated
parabolic $\mathfrak p$ is just the derived algebra of the above.
(Though it might be more appropriate to denote it by $\mathfrak
p'$, this would just be cumbersome and in any case we do not need
to refer to the parabolic itself.) We denote by $P$ the closed
subgroup of $G$ with Lie algebra $\mathfrak p$. An adapted pair
for $\mathfrak p$ was constructed in \cite {J5}. A rather unusual
(but easily proven - see \ref {1.8}) aspect of the index one case
is that such a pair is unique up to conjugation by an element of
$P$. Moreover via \cite [Cor. 8.7]{J4}, every regular orbit meets
the resulting slice at exactly one point and even transversally
(see \cite [Prop. 7.8(ii)]{J9} for example). This is the exact
analogue of the result obtained in the semisimple case by Kostant
\cite {K1}, \cite {K2}.

\subsection{}\label{1.7}

The proof that our suggestion holds in the above case is obtained
from the combinatorial analysis given in the following two
sections.  This turned out to be surprisingly difficult though
ultimately we believe the solution is rather elegant.   However
unlike the Borel case and the case $p=1$, the element in its
conclusion is obtained from a standard nilpotent element not just
by conjugation through $W$ but rather by an element of the form
$n_wb$, with $n_w$ a representative of $w \in W$ lying in $N_G(H)$
and $b \in B$.  Moreover we give a recipe for computing $w$, but
at present its meaning is unclear.

\subsection{}\label{1.8}

Define $\mathfrak p$ as in \ref {1.6}.  Let us recall the
construction of an adapted pair for $\mathfrak p$ given in \cite
{J5}.

Let $\mathfrak h'$ denote the set of diagonal matrices lying in
$\mathfrak p$.  One has $\dim \mathfrak h' = n-2$.

Let $\pi := \{\alpha_i\}_{i=1}^{n-1}$ be the set of simple roots
for $\mathfrak {sl}(n)$ labelled as in Bourbaki \cite [Planche
I]{B} with respect to the Borel subalgebra $\mathfrak b$ being the
set of upper triangular matrices of trace zero.

Identify $\mathfrak p^*$ with $\mathfrak p^-:=\kappa (\mathfrak
p)$.  Recall that we are assuming $\mathfrak b \subset \mathfrak
p^-$.  The nilradical $\mathfrak m$ of $\mathfrak p$ is a
complement to $\mathfrak p^-$ in $\mathfrak g$ and identifies via
the Killing form with the kernel of the restriction map $\mathfrak
g^* \rightarrow \mathfrak p^*$, that is $\mathfrak m =\mathfrak
k$, in the notation of \ref {1.4}. Under the present conventions
$\mathfrak m$ is spanned by those vectors corresponding to roots
in which $\alpha_p$ appears with a coefficient of $-1$.  (This
convention should be recalled in \ref {2.6} b).)

Recall (cf \cite [2.5]{J5}) the notion of the Kostant cascade
$\mathscr B$ (of positive strongly orthogonal roots) defined for
any semisimple Lie algebra. For $\mathfrak {sl}(n)$ this is just
$\{\alpha_i+\alpha_{i+1}+,\ldots,+\alpha_{n-i}\}_ {i =1}^{
[(n-1)/2]}$.  In particular $\mathscr B \cap \pi \neq \phi$, if
and only if $n$ is even and then this intersection is
$\{\alpha_{n/2}\}$.

By definition, the Levi factor of $\mathfrak p$ is isomorphic to
$\mathfrak {sl}(p) \times \mathfrak {sl}(q)$.  Let $-\mathscr B'$
denote its Kostant cascade.

Since $q,p$ are coprime, exactly one of the integers $p,q,n$ is
even and so $(\mathscr B\cup \mathscr B') \cap (\pi \cup -\pi)$
consists of exactly one element, say $\alpha$.  Set
$\overline{\mathscr B}:=\mathscr B\cup\mathscr B' \setminus
\{\alpha\}$.  One may remark that up to signs $\mathscr B\cup
\mathscr B'$ is a choice of simple roots \cite [2.6]{J5}. In
particular $\mathscr B\cup \mathscr B'$ is a basis for $\mathfrak
h^*$.

A main result of \cite {J5} is that
$$\eta:= \sum_{\beta \in \overline{\mathscr B}} x_\beta,  \eqno
{(*)}$$ is regular (see \cite [3.7]{J5}) in $\mathfrak p^*$ and
that there exists a unique $h \in \mathfrak h'$ such that $h\eta
=-\eta$, with respect to co-adjoint action.  Thus $(h,\eta)$ is an
adapted pair for $\mathfrak p$.

It is checked in \cite [3.7]{J5} that $\mathbb K x_\alpha$ is a
complement to $\mathfrak p \eta$ in $\mathfrak p^*$.  It is
further checked \cite [3.3]{J5} that $hx_\alpha=mx_\alpha$, where
$m+1$ is the degree, namely $\frac{p^2+q^2+pq-1}{2}$, of the
unique homogeneous generator $f$ of $Y(\mathfrak p)$.  Moreover
$f$ is irreducible.  Indeed otherwise it could not be the
generator of $Y(\mathfrak p)$.  We remark that for $p=1$ there is
a rather precise and elegant description of $f$ discovered
independently by Dixmier and Joseph (see \cite [15]{FJ1} and
references therein).  However even for the case $p=2,q=3$ a simple
description of $f$ is not known.

Let $\mathscr N(\mathfrak p)$, or simply $\mathscr N$, denote the
zero set of $f$ in $\mathfrak p^*$.  By the above remarks
$\mathscr N$ is an irreducible closed subvariety of codimension
one in $\mathfrak p^*$.  In more general terms if $\mathfrak a$ is
an algebraic Lie algebra, then $\mathscr N(\mathfrak a)$ is
defined to be the nilfibre of the categorical quotient map
$\mathfrak a^* \rightarrow \mathfrak a^*//A$.  It is seldom
irreducible even for a truncated biparabolic \cite [1.4]{J7}.

The relation $h\eta = -\eta$, forces $P\eta \subset \mathscr N$.
Since $P\eta$ has codimension index $\mathfrak p =1$ in $\mathfrak
p^*$, it follows from the above that $P\eta$ is open dense in
$\mathscr N$ and consequently $P\eta = \mathscr N_{reg}$.  In
particular if $(h',\eta')$ is a second adapted pair for $\mathfrak
p$, then $\eta' \in P\eta$.  Moreover since $h$ is uniquely
determined by $\eta$, we conclude that there exists $p \in P$ such
that $ph=h',p\eta =\eta'$.

Through the identifications made in \ref {1.4} we may consider
$\eta$ as an element of $\mathfrak g^*$.  Then the relation $h\eta
= -\eta$ forces $\eta \in \mathscr N(\mathfrak g)$, but one
\textit{cannot} conclude that $P\eta \subset \mathscr N(\mathfrak
g)$.   Again $\eta$ is not regular in $\mathscr N(\mathfrak g)$.
Here one should stress that by identifying $\mathfrak p^*$ with
$\kappa(\mathfrak p)$ we are choosing a particular pre-image of
$\eta \in \mathfrak g^*$.

The main result of the present paper is that there exists $\xi \in
\mathfrak m$ such that $\eta+\xi \in \mathscr N(\mathfrak
g)_{reg}$.  This can be expressed as saying that there is some
pre-image of $\eta \in \mathfrak g^*$ lying in $\mathscr
N(\mathfrak g)_{reg}$.  We do not believe that this to be a priori
obvious. However one may remark that since index $\mathfrak p =
1$, it follows that $P\eta +\mathfrak m$ has codimension $1$ in
$\mathfrak g^*$.  On the other hand since $\mathfrak g$ is
semisimple the codimension of $\mathfrak g^* \setminus \mathfrak
g^*_{reg}$ in $\mathfrak g^*$ is $3$ (as is well-known - see \cite
[2.6.14]{J6}, for example).  It follows that $(P\eta + \mathfrak
m) \cap \mathfrak g^*_{reg} \neq \phi$, however this is not quite
what we require.

\subsection{}\label{1.9}

Recall the notation of \ref {1.5}. Set $x = \sum_{\alpha \in
\pi}x_\alpha$. Set $\mathfrak n' =[\mathfrak n,\mathfrak n]$.  We
need the following well-known technical result.  We give a proof
for completeness.

\begin {lemma}    $Nx= x+\mathfrak n'$.
\end {lemma}

\begin {proof}  The inclusion $Nx\subset  x+\mathfrak n'$, is trivial.

The converse will be proved by an easy induction. For all $\beta
\in \Delta^+$, we may write $\beta = \sum_{\alpha \in \pi}
k_\alpha \alpha$ and we set $o(\beta) = \sum_{\alpha \in \pi}
k_\alpha$. Let $N_\beta$ be the closed subgroup of $N$ with Lie
algebra $\mathbb Kx_\beta$. For all $k \in \mathbb N^+$, set
$$N^k:= \prod_{\beta \in \Delta^+|o(\beta) \geq k}N_\beta. \eqno
{(*)}$$

Clearly $N^k$ is a closed subgroup of $N$ with Lie algebra
$$\mathfrak n^k:= \sum_{\beta \in \Delta^+|o(\beta) \geq k}
\mathbb Kx_\beta.$$

Set
$$\mathfrak n_k:= \sum_{\beta \in \Delta^+|o(\beta) = k}
\mathbb Kx_\beta.$$

Then $$\mathfrak n^k = \sum_{\ell \geq k} \mathfrak n_\ell.$$

Suppose we have shown that $x+\mathfrak n^{k+1} \subset N^kx$,
which is of course trivial for $k$ sufficiently large.  If $k=1$,
we are done. Otherwise use of the well-known relation $\mathfrak
n_k =[\mathfrak n_1,\mathfrak n_{k-1}]$, together with the
induction hypothesis gives the assertion for $k$ replaced by
$k-1$.

\end {proof}

\textbf{Remark 1}.  It is clear that $Nx$ is dense in $x+\mathfrak
n'$.  On the other hand $N$ is unipotent group acting linearly on
its Lie algebra $\mathfrak n$.  Thus $Nx$ is closed in $\mathfrak
n$ by a result of Rosenlicht \cite [Thm. 2]{R} and from this the
lemma follows. Actually Rosenlicht attributes (without reference)
the required version of the result to Kostant the latter having
given a ``complicated Lie algebra argument", of which we believe
the above is an extract (see \cite [Thm. 3.6]{K1}).

\

\textbf{Remark 2}.  The result is even easier for $\mathfrak
{sl}(n)$ since closure is not needed. Take $x' \in x +\mathfrak
n'$ and let $V$ be the standard $\mathfrak {sl}(n)$ module of
dimension $n$. Choose a basis in $V$ so that $x$ has Jordan block
form. From this one immediately verifies that
$V,x'V,x^{\prime2}V,\ldots,$ is a complete flag and so there
exists a basis for $V$ such that $x'$ has Jordan block form. Thus
$x'$ (as well as $x$) is regular. Consequently $\dim Nx' =\dim N
-\dim C_N(x') \geq \dim N -\dim C_G(x')= |\Delta| - \rank
\mathfrak g = \dim \mathfrak n'$. Thus $Nx'$ must be dense in the
irreducible variety $x+\mathfrak n'$ and hence open. As a special
case, $Nx$ is open dense in $x+\mathfrak n'$. Consequently these
orbits must meet and so $x' \in  Nx$.

\subsection{}\label{1.10}

Let $\mathscr P$ denote the set of pairs of coprime positive
integers and $\mathscr S$ the set of all finite ordered sequences
of ones and minus ones.  Our construction gives a map (possibly
surjective) of $\mathscr P$ into $\mathscr S$, through the
signature of a meander (see \ref {2.1} and \ref {3.3}).  We
believe this to be quite new though whether it has any
arithmetical significance is another matter.  It would be
interesting to determine the image and fibres of this map.

\subsection{}\label{1.11}

V. Popov has informed us of work of particularly the Russian
school on algebraic and rational slices.  Although this has
practically no intersection with our present paper (being
concerned mainly with case where $\mathfrak g$ is a reductive
group acting on a finite dimensional module $V$) it is
nevertheless appropriate to give a sketch of their results of
which \cite {P} provides in particular a useful summary.

Adopting the terminology of \cite {P} we call a linear action of a
Lie algebra $\mathfrak a$ on a finite dimensional vector space
co-regular if the algebra of invariant regular functions on $V$ is
polynomial.

What we call an algebraic slice in \cite [7.6]{J9}, the Russian
school had called a Weierstrass section. (This terminology comes
for the case $\mathfrak g = \mathfrak {sl}(3)$ acting on a simple
ten dimensional module for which such a section was exhibited by
Weierstrass.) The existence of a Weierstrass section (trivially)
forces the action to be co-regular.

A fairly comprehensive study of Weierstrass sections was given in
\cite [Sect 2]{P} for a co-regular action of a semisimple Lie
algebra $\mathfrak g$ acting on a finite dimensional vector space
$V$ particularly if either $\mathfrak g$ or $V$ is simple. A
notable general result \cite [Thm. 2.2.15]{P} is that a
Weierstrass section exists if the zero fibre $\mathscr
N_V(\mathfrak g)$ of the categorical quotient map admits a regular
element (in a sense analogous \ref {1.1}). Moreover the converse
holds if the set of non-regular elements in $V$ is of codimension
$\geq 2$.

The above result and the theorems leading up to it were partly
inspired by the seminal work of Kostant \cite {K1,K2} as was our
own work. Though these are stated for just $\mathfrak g$
semisimple it is not improbable that they extend to the general
case as we already partly verified in \cite {JS}. Moreover it is
interesting to note that in all our examples (with $\mathfrak g$
solvable) where we found \cite {J9} a Weierstrass section for
which the base point was not regular, the set of non-regular
elements (in $\mathfrak g^*$) was indeed of codimension 1.

Apart from these general considerations, when in comes to actually
finding a Weierstrass section the results reported in \cite {P}
and own own work \cite {J7,J8,J9} are of a quite different nature
not least because they are mainly obtained on a case by case basis
and whilst \cite [Sect 2]{P} concentrates on the semisimple case,
our own work concerns the non-reductive case. Indeed it is not
easy to find coregular actions, rather difficult to find regular
elements in the zero fibre and even harder to exhibit Weierstrass
sections if no such elements exist.  Just to exemplify the last of
these, Popov \cite [2.2.16]{P} notes that for the action
$\mathfrak {sl}(n)$ on $n$ copies of its defining $n$ dimensional
module, the invariant algebra is generated by the (obvious)
determinant and as a consequence the nilfibre has no regular
elements, whilst a Weierstrass section obtains by sending all
off-diagonal entries to zero and all diagonal entries besides the
first to zero.  On the other hand our examples \cite {J9} come
from truncated Borels of simple Lie algebras outside types
$A,C,B_{2n},F_4$ and this for the adjoint action.  In these cases
they are many generators and a Weierstrass section is not so easy
to describe.  Classifying Weierstrass sections for co-regular
actions of non-reductive groups is a wide open problem.

\textit{Acknowledgement}.  The authors would like to thank Anna
Melnikov for Latex instruction and Vladimir Popov for his comments
on some points in the manuscript.

\

A preliminary version of this result was presented by Florence
Fauquant-Millet at the Workshop "Problems and Progress in Lie
Algebraic Theory " held on 7-8 July 2010 in the Weizmann
Institute.

\section{The Combinatorial Construction}

\subsection{}\label{2.1}

Throughout $\mathfrak g$ denotes the simple Lie algebra $\mathfrak
{sl}(n)$, with $n>2$. Let $\mathfrak h$ denote the diagonal
matrices in $\mathfrak g$.  It is a Cartan subalgebra. Let $(\ ,\
)$ denote the Cartan scalar product on $\mathfrak h$.

Set $I=\{1,2,\ldots,n-1\}$, $\hat{I}=I\cup \{n\}$.  Following
Bourbaki \cite [Planche I]{B}, we choose an orthonormal basis
$\varepsilon_i:i \in \hat{I}$ in $\mathbb R^n$ and set
$\alpha_i=\varepsilon_i-\varepsilon_{i+1}:\forall i \in I$. Then
$\pi=\{\alpha_i\}_{i \in I}$ is a simple root system for
$\mathfrak g$.

Let $p,q$ be positive integers with sum $n$. We assume that $p
\leq q$.   Following a suggestion (see \cite [2.6, Remark]{J5} of
G. Binyamini we use the Dergachev-Kirillov meanders on the set
$\{\varepsilon_1,\varepsilon_2,\ldots,\varepsilon_n\}$ to describe
the support of the second element $\eta$ of the adapted pair
$(h,\eta)$ defined in \ref {1.8}. This is instead of using the
action of the group $<i,j>$ defined in \cite {FJ1} (see \cite
[2.2]{J5}) on the set $\pi:=\{\alpha_1,\alpha_2,\ldots,
\alpha_{n-1}\}$ of simple roots. Here a meander is interpreted as
an orbit of the group $\Gamma$ generated by involutions $\sigma,
\tau$ defined as follows.  For all $i=1,2,\ldots,n$, set
$\sigma(i)=n+1-i$ and
$$\tau(i)=\left\{\begin{array}{ll}p+1-i& :\ 1 \leq i \leq p ,\\
n+p+1-i & :\ p+1\leq i \leq n.\\
\end{array}\right.$$

One checks that $\tau\sigma(k)=p+k$, where it is understood that
any integer is reduced modulo $n$ so that it lies in $[1,n]$.  It
follows that $\{1,2,\ldots,n\}$ is a single $\Gamma$ orbit
$\mathscr O$ if and only if $p,q$ are coprime.

\subsection{}\label{2.2}

Assume from now on that $p,q$ are coprime.

By a slight abuse of language we say that an end point of
$\mathscr O$ is an element of $\mathscr O$ fixed by either
$\sigma$ or $\tau$. One easily checks that $\mathscr O$ has
exactly two end points $a,b$.

If $p$ is odd, we can set $a=(p+1)/2$. If in addition $n$ (resp.
$q$) is odd we can set $b=(n+1)/2$ (resp. $b=p+(q+1)/2$). In this
case $a <b$. If $p$ is even we can set $b=(n+1)/2$ and
$a=p+(q+1)/2$. In this case $b <a$.  We call $a$ (resp. $b$) the
starting (resp. finishing) point of $\mathscr O$.

We define a bijection $\varphi :\hat{I}\iso \hat{I}$ as follows.
First note that the starting point $a$ is always a $\tau$ fixed
point. Then set $\varphi(1)=a, \varphi(2)=\sigma(a),
\varphi(3)=\tau\sigma(a), \ldots$.  (This may be a little
confusing since the domain which identifies with
$\{1,2,\ldots,n\}$ and the target which identifies with $\mathscr
O$ are both denoted by $\hat{I}$.)

Set $\beta_i=\varepsilon_{\varphi(i)}-\varepsilon_{\varphi(i+1)}:i
\in I$. By our conventions $\beta_1$ is a positive (resp.
negative) root if $p$ odd (resp. even). Again
$(\beta_i,\beta_{i+1}) <0$, for all $i \in I\setminus \{n-1\}$,
whilst the remaining scalar products between distinct elements,
vanish. Hence $\Pi:=\{\beta_i\}_{i\in I}$ is a simple root system
and in particular $W$ conjugate to $\pi$.

Recall \ref {1.8}. One easily checks that \textit{up to signs}
there is a unique subset of $\Pi$ which is the Kostant cascade
$\mathscr B$ for $\mathfrak {sl}(n)$ and again \textit{up to
signs} $\Pi \smallsetminus \mathscr B$ is the opposed Kostant
cascade $\mathscr B'$ for the Levi factor $\mathfrak {sl}(p)\times
\mathfrak {sl}(q)$ of $\mathfrak p$. In particular relative to
$\pi$ the elements of $\mathscr B$ (resp. $\mathscr B'$) are
positive (resp. negative roots).  It is the analysis of these
signs which is the main combinatorial content behind the
construction of a further simple root system $\Pi^*$. This is the
main step in achieving our goal of finding a regular nilpotent
element $y$ of $\mathfrak g$, whose restriction to $\mathfrak p$
is $\eta$.

\subsection{}\label{2.3}

Towards the above goal we define a turning point of $\mathscr O$
to be an element $\varphi(t):t \in \hat{I}$ such that
$t-\sigma(t)$ is of opposite sign to $t-\tau(t)$.  Here we include
the end points of $\mathscr O$ in its set of turning points, that
is to say when one of the above integers is zero. The remaining
turning points are called internal turning points.

The observation in \ref {2.2} can be expressed as saying that for
all $i \in I$ one has either $\beta_i \in \mathscr B \cup \mathscr
B'$ or $\beta_i \in -(\mathscr B \cup \mathscr B')$. Notice
further that \textit{up to signs} if $\beta_{i-1} \in \mathscr B$,
then its successor $\beta_i \in \mathscr B'$ and vice-versa. Let
us now make precise how these signs vary. Indeed taking account of
the fact that the elements of $\mathscr B$ (resp. $\mathscr B'$)
are positive (resp. negative) roots, the following fact is easily
verified.

\begin {lemma}  Suppose $\beta_{t-1} \in \mathscr B \cup \mathscr
B'$ (resp. $\beta_{t-1} \in -(\mathscr B \cup \mathscr B')$), then
$\beta_t \in -(\mathscr B \cup \mathscr B')$ (resp. $\beta_t \in
\mathscr B \cup \mathscr B'$) if and only if $\varphi(t)$ is an
internal turning point of $\mathscr O$.
\end {lemma}

\textbf{Remark}.  By our conventions if $p$ is odd, then $\beta_1
\in \mathscr B$ and if $p$ is even, then $\beta_1 \in -\mathscr
B$.

\subsection{}\label{2.4}

It is easy to compute the set of turning points of $\mathscr O$.
They form two disjoint ``connected" sets, namely $A:=[[p/2]+1,p]$
and $B:=[[n/2]+1,p+[(q+1)/2]]$. If $p$ is odd both have
cardinality $(p+1)/2$ and moreover $a \in A$ and $b \in B$. If $p$
is even, then $A$ has cardinality $p/2$, whilst $B$ has
cardinality $1+p/2$ and contains both $a$ and $b$.

The $\Gamma$ orbit $\mathscr O$ viewed as starting at $a$ and
finishing at $b$ acquires a linear order with smallest element $a$
and largest element $b$ being the natural order on $\hat{I}$
translated under $\varphi$. It induces a linear order on the set
$T$ of turning points.

\begin {lemma} With respect to the above linear order on $T$ the nearest
neighbour(s) of an element of $A$ lie(s) in $B$ and vice-versa.
\end {lemma}

\begin {proof}  Since $|A|\leq |B|$ with equality unless both
$a,b$ lie in $B$ in which case $|A|+1=|B|$, it is enough to show
that the set of successors of an element of $B$ first meets $A$.
Take $b' \in B$ and assume that $\sigma(b')$ (resp. $\tau(b')$) is
a successor of $b'$. Since $B$ lies in an interval of width $\leq
p/2$, whilst $\tau\sigma$ (resp. $\sigma\tau$) is translation by
$p$ (resp. $-p$), it follows that the set of successors of $b'$
with respect to powers of $\tau\sigma$ (resp. $\sigma\tau$) meets
the interval $[1,p]$ before it meets $B$ again.

Finally observe that the image under $\tau$ of any of the above
$p$-translates of $b'$ (that is to say after starting at $b'$ and
until $[1,p]$ is reached) do not lie in $B$.  On the other hand
when $[1,p]$ is reached then the set of successors first meets $A$
since $A\cup\tau(A)=[1,p]$. Hence the assertion of the lemma.
\end {proof}

\subsection{}\label{2.5}

We can now describe the signs mentioned in \ref {2.3}.

In view of \ref {2.4}, that there are always $p+1$ turning points
of which $p-1$ are internal.   If $p$ is odd (resp. even) we label
them as $\varphi(t_i):i=1,2,\ldots,p+1$ (resp. $i=0,1,\ldots,p$),
where the $t_i$ are strictly increasing. Set
$J:=1,2,\ldots,[(p+1)/2]$. With this choice and our previous
conventions $A=\{\varphi(t_{2j-1}):j \in J\}$.

For all $k=0,1,\ldots,p$, set $\epsilon_i=(-1)^{k-1}$, for all
$i=t_k,t_k+1,\ldots,t_{k+1}-1$.  That is $\epsilon_i=1$ (resp.
$\epsilon_i=-1$) in the interval in which an element of $A$ (resp.
$B$) is followed by an element of $B$ (resp. $A$).

\begin {cor} $\mathscr B \cup \mathscr B' =
\{\epsilon_i\beta_i:i \in I\}$. In particular the
$\epsilon_i\beta_i:i \in I$ which lie in $\mathscr B$ (resp.
$\mathscr B'$), are positive (resp. negative) roots.
\end {cor}

\subsection{}\label{2.6}

Observe that there is exactly one index $i \in I$ such that
$\beta_i \in \pm \pi$.  We call this the exceptional index $e$ and
$\beta_e$ the exceptional value.

From the corollary we see that $\mathscr B \cup \mathscr B'$
cannot be a simple root system because successive scalar products
acquire the wrong sign as the (internal) turning points are
crossed. Our aim is to change the $\epsilon_i\beta_i$ to new
elements $\beta^*_i$, so that

\smallskip

a)  $\Pi^*:=\{\beta^*_i\}_{i\in I}$, is a simple root system,

\smallskip

b) Suppose $\beta^*_i \neq \epsilon_i\beta_i$. Then expressed as a
sum of elements of $\pi$, the "non-compact" root $\alpha_p$
appears in $\beta^*_i$ with a negative coefficient (hence with
coefficient $-1$).

\smallskip

c)  $\beta^*_e \neq \epsilon_e\beta_e$.

\smallskip

d) $\epsilon_i\beta_i \in \mathbb N \Pi^*$, for all $i \in I
\setminus \{e\}$.  Equivalently the $\epsilon_i\beta_i :i \in I
\setminus \{e\}$ are positive roots with respect to $\Pi^*$.

\smallskip

The meaning of these conditions is as follows.  Set
$$y'= \sum_{i\in I} x_{\beta^*_i}.$$

Condition a) means that $y'$ is a regular nilpotent element of
$\mathfrak {sl}(n)$ and hence can be conjugated by an element $w$
of the Weyl group $W=S_n$ to a standard nilpotent element
$y_0:=\sum_{\alpha \in \pi} x_{\alpha}$.  Condition d) means that
the $x_{\epsilon_i\beta_i}$, for $i$ non-exceptional, either
already occur in $y'$ or can be added to $y'$ as commutators of
the $ x_{\beta^*_i}:i\in I$.  In particular by Lemma \ref {1.9},
the new element $y^{\prime\prime}$ obtained by adding these
commutators, namely the $\{x_{\epsilon_i\beta_i}:
\epsilon_i\beta_i \neq \beta_i^*\}_{i \in I\setminus \{e\}}$, is
again regular nilpotent. Finally b) and c) imply that
$y^{\prime\prime}$ restricted to $\mathfrak p$ coincides with
$\eta$ as defined in \ref {1.8}.  Let $B$ denote the Borel
subgroup of $G$ defined with respect to $\pi$.

Suppose conditions a)-d) are satisfied.  We conclude by the above,
Lemma \ref {1.9} and Corollary \ref {2.5} that there exists $w\in
W$ and $b \in B$ such that the restriction of $y:=n_wby_0$ is
$\eta$. Moreover we shall give a (fairly) explicit expression for
$\Pi^*$ and this determines $w$.

Since $P$ contains the opposed Borel subgroup $B^-$ rather than
$B$ one should consider $y_0$ defined above as the negative
element of a principal s-triple.  In the present work this has no
particular significance.

The construction of $\Pi^*$ and the proof that it satisfies a)-d)
above is given in the next section.  The proof is illustrated by
Figures 1-7.  It should also be possible for the reader to
reconstruct the analysis from just these figures.

\textbf{Remarks}.  One could imagine that a simpler way to satisfy
these conditions might be possible by the following approach.
Recall that the $\epsilon_i\beta_i : i \in I$ form a basis for
$\mathfrak h^*$.  Thus we can choose $c_i \in \mathbb N^+:i \in I$
such that there is a unique element $h \in \mathfrak h$ satisfying
$h(\epsilon_i\beta_i)=c_i, \forall i \in I$, and that this element
is regular. Then $\Delta^{*+}:=\{\alpha \in \Delta| h(\alpha)>0\}$
is a choice of positive roots for $\Delta$ and so defines a set
$\Pi^*$ of simple roots in which d) will be satisfied by
construction and even in the overly strong form that
$\epsilon_i\beta_i \in \mathbb N \Pi^*$, for all $i \in I$. It is
not so obvious if and how we can choose the $c_i:i \in I$, to
ensure that b) is satisfied. Again c) will not be satisfied in
general; but in our approach we modify our solution weakening this
overly strong form of d) to recover condition c).  A postiori one
may recover a good choice of the $c_i:i \in I$ by setting
$h(\alpha)=1, \forall \alpha \in \Pi^*$.

\subsection{}\label{2.7}

Relative to $\pi$, the roots in $\mathscr B'$ have a zero
coefficient of $\alpha_p$. Thus by Corollary \ref {2.5}, all
elements of $\{\epsilon_i\beta_i\}_{i\in I}$ have a non-negative
coefficient of $\alpha_p$, which is hence in $\{0,1\}$. This
coefficient is non-zero only if $\beta_i \in \pm \mathscr B$. By
our conventions (see \ref {2.2}) $\beta_i \in \pm \mathscr B$, if
and only if $i$ is odd.  Thus $\beta_i$ has a non-zero coefficient
of $\alpha_p$ only if $i$ is odd. In particular neighbours
$\beta_i,\beta_{i+1}$ cannot both have a non-zero coefficient of
$\alpha_p$.

Fix $t \in I$. Call $t$ a nil point if the coefficient of
$\alpha_p$ in $\beta_t$ is non-zero and a boundary point if
$\varphi(t)$ or $\varphi(t+1)$ is a turning point. Call $\beta_t$
a boundary (resp. nil) value if $t$ is a boundary (resp. nil)
point. (By Corollary \ref {2.5} and our convention in \ref {1.6}
it follows that $\beta_t$ is a nil value if and only if
$x_{\epsilon_t\beta_t}$ belongs to the nilradical of $\mathfrak
p^-$.)

The unique boundary value to an end point is called an end value.

\begin {lemma}

\

(i)   Suppose $\varphi(t)$ is an internal turning point, then $t$
and $t-1$ cannot be both nil points.

(ii)   Suppose $t\in I$ is a nil boundary point with
$\varphi(t)\in B$ (resp. $\varphi (t+1)\in B$). Then $\varphi
(t+1) \in A$ (resp. $\varphi (t) \in A$).

(iii)  Suppose $\varphi(t) \in A$. Then $t-1$ or $t$ must be a nil
point, in particular $1$ must be a nil point if $p$ is odd.
\end {lemma}

\begin {proof}

(i) follows from the remarks in the first paragraph above.

(ii) Since $t$ is a nil point, we must have
$\sigma(\varphi(t))=\varphi(t+1)$.  Now suppose $i:=\varphi(t)\in
B$.  Then $i\geq [n/2]+1$.  Thus $t$ is a nil point if and only if
$\sigma(i) \leq p$.  Again $i \leq p +[(q+1)/2]$ so $\sigma(i)
=n+1-i \geq [q/2]+1 \geq [p/2]+1$. Consequently $\varphi(t+1)
=\sigma(i)\in A$, as required. The proof of the second case is
exactly the same.

(iii) Since $i:=\varphi(t) \in A$, we have $i\leq p$ and so
$\sigma(i) \geq p+1$.  Thus either $\beta_t$ or $\beta_{t-1}$ must
have a non-zero coefficient of $\alpha_p$. Hence (iii).
\end {proof}

\textbf{Remarks}.  Since
$\beta_i=\varepsilon_{\varphi(i)}-\varepsilon_{\varphi (i+1)}$ we
may regard $\beta_i$ as lying between the elements
$\varphi(i),\varphi(i+1)$ of $\varphi(\hat{I})$. We say that
$\beta_{i-1},\beta_i$ are the neighbours of $\varphi(i):i \in
\hat{I}$ in $\pm(\mathscr B \cup \mathscr B')$ and that
$\varphi(i),\varphi(i+1):i \in I$ are the neighbours of $\beta_i
\in \pm(\mathscr B \cup \mathscr B')$. Then (iii) of the lemma can
be expressed as saying that every element of $A$ has exactly one
nil boundary value neighbour, whereas (ii) of the lemma can be
expressed as saying that if a nil boundary value has an element of
$B$ as a neighbour, then it is sandwiched between an element of
$A$ and an element of $B$.  By (ii) and (iii) of the lemma every
nil boundary value has a unique element of $A$ as a neighbour.
Finally an end value is non-nil only if it is the (unique)
neighbour of an element of $B$. For example if $p=2,q=5$, both
end-points are non-nil.  However an end value can be nil even if
it is a neighbour of an element of $B$.  For example if $p=2,q=3$,
the starting value is nil.

\subsection{Intervals}\label{2.8}

Let $\varphi(s),\varphi(t) \in T$ be turning points with $s<t$.
The subset $I_{s,t}:=\{s,s+1,\ldots,t-1\}$ is called an interval.
If $\varphi(t)$ is the immediate successor to $\varphi(s)$ in $T$,
it is called a simple interval.  Otherwise it is called a compound
interval.

The sum
$$\iota_{s,t}:=\sum_{i\in I_{s,t}}\beta_i,$$
is called a simple (resp. compound) interval value if $I_{s,t}$ is
simple (resp. compound).  The set $\{\beta_i: i \in I_{s,t}\}$ is
called the support of $\iota_{s,t}$ or of $I_{s,t}$. 

\begin {lemma} Let $I_{r,s}$ be a simple interval.
There is exactly one $i \in I_{r,s}$ such that $\beta_i$ has a
non-zero coefficient of $\alpha_p$.

\begin {proof}  Observe that
$$\iota_{r,s}=\varepsilon_{\varphi(r)}-\varepsilon_{\varphi(s)}.\eqno{(*)}$$
Since $\varphi(r)\in A$ and $\varphi(s )\in B$ or vice-versa, it
follows from \ref {2.4} that the coefficient of $\alpha_p$ in the
above sum equals one or minus one. Moreover there can be no
cancellations of coefficients of $\alpha_p$ in the sum because
only alternate terms can have a non-zero coefficient and these are
either all positive roots or all negative roots since the indices
lie between successive turning points.  Hence the assertion.
\end {proof}
\end {lemma}

\subsection{The Sign of Simple Interval Values}\label{2.9}

Take $\varphi(s) \in T$ and let $\varphi(t)$ be its immediate
successive in $T$. The sign of the simple interval value
$\iota_{s,t}$ is said to be positive (resp. negative) if
$\varphi(s) \in A$ (resp. $\varphi(s) \in B$).

A positive (resp. negative) interval value is a positive (resp.
negative) root relative to $\pi$ with the coefficient of
$\alpha_p$ being $1$ (resp. $-1$), by the definition of the
$\epsilon_i$, Corollary \ref {2.5} and Lemma \ref {2.8}.

\subsection{The Exceptional Value}\label{2.10}

Recall the exceptional value $\beta_e:e \in I$ defined in \ref
{2.6}. It defines a unique simple root $\alpha \in \pi$ and one
has $\beta_e = \alpha$, \textit{up to a sign}.

\begin {lemma} The exceptional value $\beta_e$ is never a
nil value, equivalently $\alpha\neq \alpha_p$.
It is a boundary value to some unique turning point, which is
either

(i) internal,

or

(ii) an end point lying in $B$.
\end {lemma}

\begin {proof} Since $p,q$ are coprime, exactly one of the
integers $p,2p+q,n=p+q$, call it $m$, is even. Then $\beta_e=\pm
\alpha_{m/2}$.  For this to be a nil value we would need
$\alpha_{m/2}=\alpha_p$, that is $m/2 = p$, which is impossible
since $p < n/2$.  This proves the first assertion.

One checks from the description of the turning points in \ref
{2.4} that either $m/2$ or $1+m/2$ is a turning point. They cannot
both be turning points because then $\beta_e$ would be nil by \ref
{2.8}, contradicting the first part. On the other hand by Lemma
\ref {2.7}(iii) an end-point lying in $A$ is nil. Hence the second
assertion.
\end {proof}

\textbf{Remark}.  On may check from \ref {2.2}, that (ii) holds if
and only if $p=1$.

\subsection{Isolated values}\label{2.11}

We call $t \in I$ an isolated point if both $\varphi(t)$ and
$\varphi(t+1)$ are turning points.  By Lemma \ref {2.8} an
isolated point is necessarily nil.  If $t$ is an isolated point we
call $\beta_t$ an isolated value. 

\section{The Description of $\Pi^*$}

\subsection{}\label{3.1}

Recall that we aim to construct $\Pi^*$ by changing some of the
$\epsilon_i\beta_i$.  Here it is convenient to write
$\beta^*_i=\epsilon_i\beta'_i$ and to say that a value is changed
if $\beta_i' \neq \beta_i$.  Let us describe those values that are
changed. Here we impose three general rules. The first two are

\smallskip

1)  Change only boundary values and change only those which are
non-nil.

\smallskip

2)  Change exactly one of the boundary values at each internal
turning point.

\smallskip

\textbf{Remark}.  In the initial stage (up to \ref {3.7}) end
values will not be changed.  However if an end value is
exceptional (and hence non-nil and so the unique neighbour of an
element of $B$), then it will be changed in the final stage (\ref
{3.8}).

\subsection{}\label{3.2}

To describe our third general rule we need the following
preliminary.  For the moment we ignore the exceptional value.

For all $i,j \in \hat{I}$, with $i<j$, set
$\beta_{i,j}=\varepsilon_{\varphi(i)}-\varepsilon_{\varphi(j)}$,
which we recall is a root (and positive with respect to $\Pi$). In
this notation $\beta_i=\beta_{i,i+1}$.  Again if
$\varphi(i),\varphi(j) \in T$, we have $\beta_{i,j}=\iota_{i,j}$.

Now let $\beta_i$ be a boundary value which is to be changed
(according to rules 1) and 2) - in particular $\beta_i$ is a
non-nil boundary value). Then there is an internal turning point,
say $\varphi(t_s)$ which either equals $\varphi(i+1)$ or
$\varphi(i)$. (Both possibilities cannot simultaneously arise
since otherwise by Lemma \ref {2.8}, $\beta_i$ would be a nil
boundary value.) In the first (resp. second) case we shall say
that $\beta_i$ is above (resp. below) $\varphi(t_s)$.

In the first case we replace $\beta_i$ by $\beta_i'$ defined by
adding to $\beta_i$ ``a compound interval value which is an odd
sum of simple interval values below $\varphi(t_s)$", that is to
say we set
$$\beta_i'=\beta_i+\beta_{t_s,t_r}=\beta_{t_
s-1,t_r}:r-s \in 2\mathbb N + 1.\eqno{(*)}$$

In the second case we replace $\beta_i$ by $\beta_i'$ defined by
adding to $\beta_i$ ``a compound interval value which is an odd
sum of simple interval values above $\varphi(t_s)$", that is to
say we set
$$\beta_i'=\beta_i+\beta_{t_r,t_s}=\beta_{t_r,t_s+1}:s-r \in 2\mathbb N + 1.\eqno{(**)}$$

We may summarize the above by saying that in both cases the added
interval value is on the opposite side of the turning point to the
element in question and is a sum of an odd number of simple
interval values.

Finally (recall) that we set $\beta^*_i=\epsilon_i\beta_i'$.

\begin {lemma}  Suppose $\beta_i' \neq \beta_i$.  Then $\epsilon_i\beta_i'$ is a root. Moreover expressed
as a sum of elements of $\pi$, the non-compact root $\alpha_p$ has
coefficient $-1$ in $\epsilon_i\beta_i'$.

\end {lemma}

\begin {proof} Since $\beta_i$ is a non-nil boundary value
the coefficient of $\alpha_p$ in it is zero.

By definition $\epsilon_i$ changes sign as each turning point is
crossed. Thus if $\epsilon_i =1$ (resp. $\epsilon_i =-1$). then
the nearest simple interval value added to $\beta_i$ is negative
(resp. positive). Then by the second paragraph of \ref {2.9} the
coefficients of $\alpha_p$ of the successive simple interval
values added to $\epsilon_i\beta_i$ are $\{-1,1,-1,\ldots\}$,
whereas by construction the number of such simple intervals is
odd.
\end {proof}

\subsection{Signature}\label{3.3}

To complete our description of $\Pi^*$ we must now specify which
boundary values are to be changed (which will specify $t_s$ and
how $t_r$ in equations $(*)$ and $(**)$ is determined).

The above data will be completely determined by the signature of
the orbit $\mathscr O$ defined as follows.

Recall that by the choices made in \ref {2.5} and by Lemma \ref
{2.7}, at each turning point $\varphi(t_{2j-1}):j \in J$, which we
recall lies in $A$, one has that either $\beta_{t_{2j-1}-1}$ or
$\beta_{t_{2j-1}}$ is nil (but not both). In the first case we set
sg$(j)=-1$ (to specify that the nil boundary value is above the
turning point) and in the second case we set sg$(j)=1$ (to specify
that the nil boundary value is below the turning point).

If sg$(1)=1$ (which is always the case if $p$ is odd) then there
is a unique increasing sequence $j_1,j_2, \ldots,j_r\in J$ with
$j_1=1$, such that
$$\text {sg}(i)=(-1)^{u-1}, \forall i= j_u, j_u+1, \ldots, j_{u+1}-1,
\forall u=1,2,\ldots,r-1, \quad \text {sg}(j_r)=(-1)^{r-1}.$$

If sg$(1)=-1$ (which can be the case if $p$ is even) then there is
a unique increasing sequence $j_1,j_2, \ldots,j_r\in J$ with
$j_1=1$, such that
$$\text {sg}(i)=(-1)^u, \forall i= j_u, j_u+1, \ldots, j_{u+1}-1,
\forall u=1,2,\ldots,r-1, \quad \text {sg}(j_r)=(-1)^r.$$

We say that the signature at the turning point
$\varphi(t_{2j-1})\in A$ is positive (resp. negative) if sg$(j)=1$
(resp. sg$(j)=-1$).

Finally the signature of $\mathscr O$ is defined to be the set
$\{sg(i)\}_{i=1}^{[p/2]}$.  In the notation of \ref {1.10}, it
lies in $\mathscr S$ and defines a map of the set of coprime pairs
$\mathscr P$ into $\mathscr S$.

\subsection{}\label{3.4}

We assume until the end of \ref {3.8}, that the signature at the
first turning point in $A$ is positive. This is always the case if
$p$ is odd by virtue of Lemma \ref {2.7}(iii). Then the easiest
case to describe is when there are no signature changes. This is
illustrated in Figure $1$, where the given pattern is repeated as
many times as there are turning points in $A$.

\begin{center}

\begin{picture}(200,240)(-60,-50)
\put(50,-50){\line(0,1){240}}
\multiput(73,160)(6,0){8}{\line(1,0){3}} \put(127,157){$c_1$}
\multiput(-26,150)(6,0){13}{\line(1,0){3}}
\put(-70,147){$\varphi(t_{2j-1})$} \put(55,147){$A$}
\multiput(-26,50)(6,0){13}{\line(1,0){3}}
\put(-62,47){$\varphi(t_{2j})$} \put(55,-14){$A$}
\multiput(-26,-10)(6,0){13}{\line(1,0){3}}
\put(-70,-13){$\varphi(t_{2j+1})$} \put(55,46){$B$}
\put(50,140){\circle{7}} \put(50,20){\circle{7}}
\multiput(93,40)(6,0){5}{\line(1,0){3}} \put(127,37){$c_2$}
\linethickness{0.4mm} \put(10,160){\line(1,0){60}}
\put(30,140){\line(1,0){60}} \put(30,120){\line(1,0){40}}
\put(30,100){\line(1,0){40}} \put(30,80){\line(1,0){40}}
\put(10,60){\line(1,0){60}} \put(30,40){\line(1,0){60}}
\put(30,20){\line(1,0){40}} \put(30,0){\line(1,0){40}}
\put(10,60){\line(0,1){100}} \put(30,120){\line(0,1){20}}
\put(30,80){\line(0,1){20}} \put(30,20){\line(0,1){20}}

\put(70,100){\line(0,1){20}} \put(70,60){\line(0,1){20}}
\put(70,00){\line(0,1){20}} \put(90,40){\line(0,1){100}}
\multiput(50,-10)(0,20){9}{\circle*{3}}
\multiput(70,160)(0,4){6}{\line(0,1){2}}
\multiput(30,-22)(0,4){6}{\line(0,1){2}}
\end{picture}

\end{center}
\begin{center}
{\it Figure 1.} \\
\end{center}

\textit{This shows the basic repeating pattern for positive
signature. Turning points are labelled by their type, that is $A$
or $B$, and nil values are encircled. The dots on the vertical
central line label a subset of $\hat{I}=\{1,2,\ldots,n\}$.  In the
language of \ref {3.4}, the turning point $\varphi(t_{2j-1}) \in
A$ has positive signature. In the terminology of \ref {3.6}, the
thickened lines describe the links between the elements of
$\beta^*_i:i \in I$ and define a sub-chain linking
$\beta^*_{t_{2j-1}-1}$ to $\beta^*_{t_{2j}}$. The non-nil boundary
values that are changed carry the symbol $c$ which is given the
subscript $1$ or $2$ depending on whether rule 1) or 2) of \ref
{3.1} is applied as described in \ref {3.5}.  The map $\chi$
defined in \ref {3.5} takes $\varphi(t_{2j-1})$ to
$\varphi(t_{2j})$. A mirror reflection perpendicular to the
vertical axis gives the basic repeating pattern for negative
signature.}

\bigskip

The next easiest case is when there is one signature change,
namely from positive to negative. After that there is the case
when there are two signature changes, namely from positive to
negative to positive. From then on it is simply a repetition of
the procedure for two signature changes. The first two cases can
be considered as a degeneration of the third by simply eliminating
terms. Thus it will suffice to describe the case of two signature
changes, to make precise what is meant by degeneration and to
describe the modification needed if the signature of $\mathscr O$
is initially negative which can happen if $p$ is even.

Fix a positive odd integer $u\leq r$ and set $j = j_u, k=j_{u+1},
\ell =j_{u+2}$.  The first and second cases above correspond to
$k$ not being defined and $k$ being defined but $\ell$ not being
defined.

Assume $k$ is defined.  Then by definition
$\varphi(t_{2(k-1)-1})\in A$ and admits a nil boundary value just
below, namely $\beta_t:t=t_{2(k-1)-1}$. If $\ell$ is not defined
let $\varphi(s)$ be the last turning point or (end point)
$\varphi(n)$, otherwise set $s=t_{2\ell-2}$. In both cases
$\varphi(s) \in B$.

If $t>1$, set
$$\beta'_{t-1}=\beta_{t-1} +\beta_{t,s}=\beta_{t-1,s}.\eqno{(*)}$$

Observe that $\beta_{t,s}$ is defined even when $t=1$ and is a
compound interval value $\iota_{t,s}$ which is a sum of
$2m+1:m=\ell-k$ adjacent simple interval values
$\iota_1,\iota_2,\ldots, \iota_{2m+1}$, starting at
$\iota_1=\iota_{t_{2k-3},t_{2k-2}}$.

Suppose $\ell$ is defined (and hence so is $\beta_s$).  If
$\iota_1$ is not reduced to an isolated value, set

$$\beta'_s=\beta_s +\beta_{t,s}=\beta_{t,s+1}.\eqno{(**)}$$

Otherwise set

$$\beta'_s=\beta_s +\sum_{i=3}^{2m+1}\iota_i=\beta_{t_{2k-1},s+1}.\eqno{(***)}$$

\subsection{}\label{3.5}

In the remaining cases added interval values will be simple and we
only have to specify the non-nil boundary values which are
changed. We need only describe these between the turning points
$t_{2j-1}$ and $t_{2\ell-1}$ in the notation of \ref {3.4} since
the pattern just repeats itself.  For this there is a simple
algorithm.

Notice first that if a simple interval value $\iota$ is to be
added to a non-nil boundary value $\beta_i$, then this simple
interval value is uniquely determined by $\beta_i$ itself via the
rule in \ref {3.2}.

We define a map $\chi$ from the set of internal turning points of
$A$ to the set of turning points of $B$.

At every internal turning point $\varphi(t) \in A$, (so then
$t=t_{2v-1}$ for some $v \in J$) there is exactly one non-nil
boundary value $\beta_u$ and by rule 1) of \ref {3.1}, it must be
changed, that is $\beta'_u\neq \beta_u$. By the rule described in
\ref {3.2} there is a unique turning point $\varphi(t') \in B$ so
that $\beta'_u=\beta_u+\iota_{t,t'}$, where we have defined
$\iota_{t,t'}:=\iota_{t',t}$ if $t'<t$. We set $t'=\chi(t)$ and
$\chi(\varphi(t))=\varphi(\chi(t))=\varphi(t)$. Notice that $t'>t$
(resp. $t<t'$) if the signature of $\varphi(t) \in A$ is positive
(resp. negative) and we say that $\varphi(t') \in B$ is a
subsequent (resp. previous) turning point to $\varphi(t) \in A$.

Observe further that there is a unique boundary value $\beta_w$ to
$\varphi(t') \in B$ such that the scalar product
$(\beta_u',\beta_w)$ is strictly positive. With one possible
exception (within a double signature change) described below, we
set $\beta_w'=\beta_w$, that is this particular boundary value is
left unchanged. Then if $b':=\varphi(t')$ is an internal turning
point, its second boundary value $\beta_{w'}$ (and for which
$(\beta_u',\beta_{w'})$ is strictly negative) should be changed by
rule 2) of \ref {3.1} \textit{unless it is a nil boundary value}.
Let us show that the latter can occur at most once (within a
double signature change). It results from $\iota_1$ (as defined in
\ref {3.4}) being reduced to an isolated value.  In this case we
shall compute $w'$ explicitly.

Suppose that $\beta_{w'}$ is a nil boundary value. Then by Lemma
\ref {2.7}(ii) it is an isolated value.

Suppose its second neighbour $a'\in A$ lies above $b'$, so then
$w'=t'-1,w=t'$. This means that the signature at $a'=\varphi(w')$
is positive. Then $t'= \chi(t)$ implies that $\varphi(t) \in A$ is
a subsequent turning point to $b'$ with negative signature.  By
definition of $k$ this forces $t=t_{2k-1}$. Consequently
$t'=t_{2k-2}$, and then $\beta_{w'}=\iota_1$, by definition of the
latter.  In particular $\iota_1$ is reduced to an isolated value.
Then we set $\beta'_{w'}=\beta_{w'}$ and
$\beta_w'=\beta_w+\iota_1$.

This is the only case (within a double signature change) that we
leave unchanged the unique neighbour $\beta_{w'}$ of $\varphi(t')$
for which $(\beta_u',\beta_{w'})$ is strictly negative.

A similar argument to the above shows that $a'$ cannot lie below
$b'$.  Indeed this would imply that the signature of $a'$ is
negative, whilst $t'=\chi(t)$ implies that $\varphi(t) \in A$ is a
previous turning point to $b'$ with positive signature.  However
the construction of \ref {3.4} has the property that if
$\varphi(t) \in A$ has positive signature then the immediate
subsequent turning point in $A$ to $\varphi(\chi(t))$, which is
$a'$ in the present application, has positive signature, so unlike
the previous case we obtain to a contradiction.

Figure $2$ compares the cases when $\iota_1$ is not and is reduced
to an isolated value.   From it one may see why \ref {3.4}$(**)$
has been replaced by \ref {3.4}$(***)$.

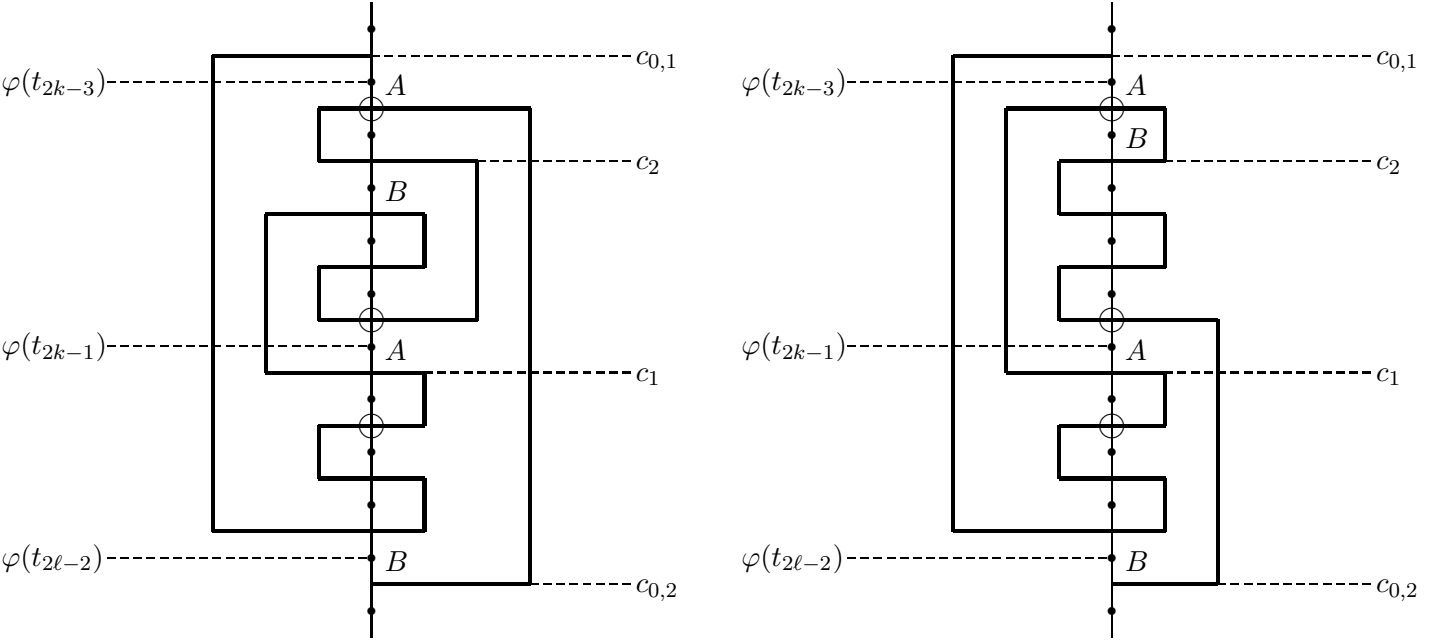
\begin{center}

\begin{picture}(500,250)(-10,-50)
\put(20,-50){\line(0,1){240}} \put(300,-50){\line(0,1){240}}
\multiput(-80,160)(5,0){20}{\line(1,0){3}}
\multiput(-80,60)(5,0){20}{\line(1,0){3}}
\multiput(-80,-20)(5,0){20}{\line(1,0){3}}
\multiput(20,170)(5,0){20}{\line(1,0){3}}
\multiput(20,130)(5,0){20}{\line(1,0){3}}
\multiput(20,-30)(5,0){20}{\line(1,0){3}}
\multiput(20,50)(5,0){20}{\line(1,0){3}}
\multiput(200,160)(5,0){20}{\line(1,0){3}}
\multiput(200,60)(5,0){20}{\line(1,0){3}}
\multiput(200,-20)(5,0){20}{\line(1,0){3}}
\multiput(300,170)(5,0){20}{\line(1,0){3}}
\multiput(300,130)(5,0){20}{\line(1,0){3}}
\multiput(300,-30)(5,0){20}{\line(1,0){3}}
\multiput(300,50)(5,0){20}{\line(1,0){3}}
\put(120,167){$c_{0,1}$}\put(120,127){$c_2$}\put(120,47){$c_1$}\put(120,-33){$c_{0,2}$}
\put(400,167){$c_{0,1}$}\put(400,127){$c_2$}\put(400,47){$c_1$}\put(400,-33){$c_{0,2}$}
\put(25,155){$A$} \put(25,55){$A$} \put(305,155){$A$}
\put(305,55){$A$}
\put(-120,157){$\varphi(t_{2k-3})$}\put(-120,57){$\varphi(t_{2k-1})$}
\put(-120,-23){$\varphi(t_{2\ell-2})$}
\put(160,157){$\varphi(t_{2k-3})$}\put(160,57){$\varphi(t_{2k-1})$}
\put(160,-23){$\varphi(t_{2\ell-2})$}
 \put(25,115){$B$} \put(305,-25){$B$}
\put(305,135){$B$} \put(25,-25){$B$}
\put(20,150){\circle{9}}\put(20,70){\circle{9}}\put(20,30){\circle{9}}
\put(300,150){\circle{9}}\put(300,70){\circle{9}}\put(300,30){\circle{9}}
\linethickness{0.4mm} \put(-40,170){\line(1,0){60}}
\put(-40,-10){\line(1,0){80}} \put(0,150){\line(1,0){80}}
\put(0,130){\line(1,0){60}} \put(-20,110){\line(1,0){60}}
\put(0,90){\line(1,0){40}} \put(0,70){\line(1,0){60}}
\put(-20,50){\line(1,0){60}} \put(0,30){\line(1,0){40}}
\put(0,10){\line(1,0){40}} \put(20,-30){\line(1,0){60}}
\put(280,30){\line(1,0){40}} \put(260,150){\line(1,0){60}}
\put(240,170){\line(1,0){60}} \put(240,-10){\line(1,0){80}}
\put(280,130){\line(1,0){40}} \put(280,110){\line(1,0){40}}
\put(280,90){\line(1,0){40}} \put(280,30){\line(1,0){40}}
\put(280,10){\line(1,0){40}} \put(300,-30){\line(1,0){40}}
\put(280,70){\line(1,0){60}} \put(260,50){\line(1,0){60}}
\put(-40,-10){\line(0,1){180}} \put(80,-30){\line(0,1){180}}
\put(0,130){\line(0,1){20}}
\put(0,70){\line(0,1){20}}\put(0,10){\line(0,1){20}}
\put(40,90){\line(0,1){20}}\put(40,30){\line(0,1){20}}\put(40,-10){\line(0,1){20}}
\put(60,70){\line(0,1){60}} \put(-20,50){\line(0,1){60}}
\put(240,-10){\line(0,1){180}} \put(260,50){\line(0,1){100}}
\put(280,110){\line(0,1){20}}
\put(320,130){\line(0,1){20}}\put(320,90){\line(0,1){20}}
\put(280,70){\line(0,1){20}}\put(320,30){\line(0,1){20}}
\put(280,10){\line(0,1){20}}\put(320,-10){\line(0,1){20}}
\put(340,-30){\line(0,1){100}}
\multiput(20,-40)(0,20){12}{\circle*{3}}
\multiput(300,-40)(0,20){12}{\circle*{3}}
\end{picture}

\end{center}
\begin{center}
{\it Figure 2.}
\end{center}
\textit{In the notation of \ref {3.4} this compares the cases when
$\iota_1$ is not isolated (on the left) and $\iota_1$ is isolated
(on the right). The same conventions as in Figure 1 apply, where
in addition the additional subscript $0$ to $c_1,c_2$ refers to
the application of the rule described in \ref {3.4}. The
particular case here corresponds to taking $\ell =k+1$. For the
general case one must extend the two outermost lines in each
diagram downwards and insert a further $\ell-(k+1)$ copies of the
basic repeating pattern for negative signature. This is
illustrated by Figure 3 in which the case $\ell=k+2$ is
considered.}

\bigskip

We remark that one may have an isolated point in a region of
negative signature.  This is illustrated in Figure $4$.

\smallskip

One checks from \ref {3.4}, \ref {3.5} that the map $\chi$ defined
above is an injection from the set of internal turning points
lying in $A$ to the set of turning points in $B$. In all cases the
latter set has cardinality one greater than the former.  Thus the
cokernel of $\chi$ is a singleton which we call the undecided
element $d \in B$. It is clear that the above algorithm just
leaves at most one boundary value of $d$ undecided. The exact
location of $d$ depends on the signature of $\mathscr O$ as we now
explain.

Suppose $p$ is odd and recall $\varphi(t_1)$ has positive
signature.  If $\varphi(t_3)$ is not defined, then
$d=\varphi(t_2)$. It is a finishing point in $B$ and we leave its
unique neighbour unchanged. If $\varphi(t_3)$ is defined and has
positive signature, then $d=\varphi(t_2)$ and we change
$\beta_{t_2}$ by adding a simple interval value, namely
$\iota_{t_1,t_2}$.

Finally suppose that $\varphi(t_3)$ has a negative signature. This
corresponds to having $k=2$ in \ref {3.4}.  Then in the notation
of \ref {3.4} one has $d=\varphi(s)$. If $d$ is a finishing point
we leave its unique neighbour unchanged.  Otherwise we change
$\beta_s$ by the rules described in \ref {3.4}$(***)$ or \ref
{3.4}$(**)$, depending on whether $\iota_1$ is reduced to an
isolated value or not.

\smallskip

\textit{One may remark that when $\varphi(t_3)$ is defined the
solutions given in the above two paragraphs would result if we
were to treat $\varphi(t_1)$ as if it were an internal turning
point.}

\smallskip

Suppose $p$ is even.  If $\varphi(t_3)$ has positive signature,
Then $d=\varphi(1)$ and we leave its unique neighbour unchanged.
The case when $\varphi(t_3)$ has negative signature will be
postponed to \ref {3.9}.

Note that if $\varphi(t')$ is an end-point, namely $t'=1$ (resp.
$t'=n$), its unique boundary value, that is $\beta_1$ (resp.
$\beta_{n-1}$) is left unchanged by the above procedure. However
this will be modified in \ref {3.8}.

This (nearly !) completes our description of $\Pi^*$.  What can
happen however is that condition c) of \ref {2.6} can sometimes
fail and this will need a further modification to be described in
\ref {3.8}.

\subsection{}\label{3.6}

Define $\beta'_i:i \in I$ through the rules described in \ref
{3.1} - \ref {3.5}, set $\beta_i^*=\epsilon_i \beta'_i$ and
$\Pi^*=\{\beta^*_i\}_{i\in I}$.

Here we show that $\Pi^*$ satisfies condition a) of \ref {2.6} by
exhibiting an  ordering so that nearest neighbours have a strictly
negative scalar product (which we call a link) and that no other
non-zero scalar products exists between distinct elements.  We
call such a succession of links, a sub-chain.

In this we shall assume that $k$ and $\ell$ of \ref {3.4} are
defined, otherwise one just obtains a degeneration of that case.

Retain the notation of \ref {3.4}.  Our construction gives a link
between $\beta^*_{t_{2j-1}-1}$ (if it is defined) and
$\beta^*_{t_{2j}-1}$ which is connected via a sub-chain to the
elements in the support of $I_{t_{2j-1},t_{2j}}$ taken in the
\textit{reverse} order.  Furthermore the last element in this
chain, namely $\beta^*_{t_{2j-1}}$ is linked to $\beta^*_{t_{2j}}$
which is connected via a sub-chain of elements in the support of
$I_{t_{2j},t_{2j+1}}$ taken in their natural order to
$\beta^*_{t_{2j+1}-1}$, by repeating the pattern in Figure $1$ the
appropriate number of times. This process is repeated till one
reaches $\beta^*_{t_{2k-3}-1}$ which lies just above the last
turning point in $A$ with positive signature. Thus
$\beta^*_{t_{2j-1}-1}$ is connected via a sub-chain to
$\beta^*_{t_{2k-3}-1}$.

A similar (reversed) phenomenon occurs in a region of negative
signature.  In particular $\beta^*_{t_{2k-1}}$ is connected via a
sub-chain to $\beta^*_{t_{2\ell-2}-1}$, which is in turn linked
via \ref {3.4}$(*)$ to $\beta^*_{t_{2k-3}-1}$. This can be
illustrated by simply making a mirror reflection of Figure $1$
perpendicular to its main axis (in simple language turning it
upside down).

If $\iota_1$ is not reduced to an isolated value, then
$\beta^*_{t_{2k-1}}$ is connected via a sub-chain to
$\beta^*_{t_{2k-3}}$ which is in turn linked to
$\beta^*_{t_{2\ell-2}}$ via \ref {3.4}$(**)$, the latter being
connected by a sub-chain to $\beta^*_{t_{2\ell-1}-1}$.

Thus using a line to designate a link or a sub-chain we may
summarize the above as
$$\beta^*_{t_{2j-1}-1}-\beta^*_{t_{2k-3}-1}-\beta^*_{t_{2\ell-2}-1}-
\beta^*_{t_{2k-1}}-\beta^*_{t_{2k-3}}-\beta^*_{t_{2\ell-2}}-\beta^*_{t_{2\ell-1}-1}.
\eqno {(*)}$$

Except for the two extreme terms the links or sub-chains between
these elements are illustrated in the left hand side of Figure
$3$.

\begin{center}

\begin{picture}(600,400)(-10,-170)
\put(20,-165){\line(0,1){360}} \put(300,-165){\line(0,1){360}}
\multiput(-80,160)(5,0){20}{\line(1,0){3}}
\multiput(-80,60)(5,0){20}{\line(1,0){3}}
\multiput(-80,-140)(5,0){20}{\line(1,0){3}}
\multiput(20,170)(5,0){20}{\line(1,0){3}}
\multiput(20,150)(5,0){20}{\line(1,0){3}}
\multiput(20,-130)(5,0){20}{\line(1,0){3}}
\multiput(300,-150)(5,0){20}{\line(1,0){3}}
\multiput(300,-130)(5,0){20}{\line(1,0){3}}
\multiput(20,-150)(5,0){20}{\line(1,0){3}}
\multiput(20,50)(5,0){20}{\line(1,0){3}}
\multiput(200,160)(5,0){20}{\line(1,0){3}}
\multiput(200,60)(5,0){20}{\line(1,0){3}}
\multiput(200,-140)(5,0){20}{\line(1,0){3}}
\multiput(300,170)(5,0){20}{\line(1,0){3}}
\multiput(300,150)(5,0){20}{\line(1,0){3}}
\multiput(300,70)(5,0){20}{\line(1,0){3}}
\multiput(300,130)(5,0){20}{\line(1,0){3}}
\multiput(300,50)(5,0){20}{\line(1,0){3}}
\put(120,167){$\beta_{t_{2k-3}-1}^*$}
\put(400,167){$\beta_{t_{2k-3}-1}^*$}
\put(400,127){$\beta_{t_{2k-2}}^*$}

\put(120,147){$\beta_{t_{2k-3}}^*$}
\put(400,147){$\beta_{t_{2k-3}}^*$}
\put(120,-133){$\beta^*_{t_{2\ell-2}-1}$}
\put(120,-153){$\beta^*_{t_{2\ell-2}}$}
\put(400,-133){$\beta^*_{t_{2\ell-2}-1}$}
\put(400,-153){$\beta^*_{t_{2\ell-2}}$}
\put(120,47){$\beta^*_{t_{2k-1}}$}
\put(400,47){$\beta^*_{t_{2k-1}}$}
 \put(400,67){$\beta^*_{t_{2k-1}-1}$}

\put(25,155){$A$} \put(25,55){$A$} \put(305,-65){$A$}
\put(25,-65){$A$} \put(305,155){$A$} \put(305,55){$A$}
\put(-120,157){$\varphi(t_{2k-3})$}\put(-120,57){$\varphi(t_{2k-1})$}
\put(-120,-143){$\varphi(t_{2\ell-2})$}
\put(160,157){$\varphi(t_{2k-3})$}\put(160,57){$\varphi(t_{2k-1})$}
\put(160,-143){$\varphi(t_{2\ell-2})$}
 \put(25,115){$B$} \put(305,-25){$B$}
\put(305,135){$B$} \put(25,-25){$B$} \put (25,-145){$B$}\put
(305,-145){$B$}
\put(20,150){\circle{9}}\put(20,70){\circle{9}}\put(20,30){\circle{9}}
\put(20,-50){\circle{9}}\put(20,-90){\circle{9}}
\put(300,150){\circle{9}}\put(300,70){\circle{9}}\put(300,30){\circle{9}}
\put(300,-50){\circle{9}}\put(300,-90){\circle{9}}
\linethickness{0.4mm} \put(-40,170){\line(1,0){60}}
\put(-40,-130){\line(1,0){80}} \put(0,150){\line(1,0){80}}
\put(0,130){\line(1,0){60}} \put(-20,110){\line(1,0){60}}
\put(0,90){\line(1,0){40}} \put(0,70){\line(1,0){60}}
\put(-20,50){\line(1,0){60}} \put(0,30){\line(1,0){40}}
\put(0,10){\line(1,0){40}} \put(20,-150){\line(1,0){60}}
\put(280,30){\line(1,0){40}} \put(260,150){\line(1,0){60}}
\put(240,170){\line(1,0){60}} \put(240,-130){\line(1,0){80}}
\put(280,130){\line(1,0){40}} \put(280,110){\line(1,0){40}}
\put(280,90){\line(1,0){40}} \put(280,30){\line(1,0){40}}
\put(280,10){\line(1,0){40}} \put(300,-150){\line(1,0){40}}
\put(280,70){\line(1,0){60}} \put(260,50){\line(1,0){60}}
\put(-20,-10){\line(1,0){60}}\put(260,-10){\line(1,0){60}}
\put(0,-30){\line(1,0){40}}\put(280,-30){\line(1,0){40}}
\put(0,-70){\line(1,0){40}}\put(280,-70){\line(1,0){40}}
\put(0,-90){\line(1,0){40}}\put(280,-90){\line(1,0){40}}
\put(0,-110){\line(1,0){40}}\put(280,-110){\line(1,0){40}}
\put(-20,-50){\line(1,0){60}}\put(260,-50){\line(1,0){60}}
\put(-40,-130){\line(0,1){300}} \put(80,-150){\line(0,1){300}}
\put(0,130){\line(0,1){20}}
\put(0,70){\line(0,1){20}}\put(0,10){\line(0,1){20}}
\put(40,90){\line(0,1){20}}\put(40,30){\line(0,1){20}}\put(40,-10){\line(0,1){20}}
\put(60,70){\line(0,1){60}} \put(-20,50){\line(0,1){60}}
\put(240,-130){\line(0,1){300}} \put(260,50){\line(0,1){100}}
\put(280,110){\line(0,1){20}}
\put(320,130){\line(0,1){20}}\put(320,90){\line(0,1){20}}
\put(280,70){\line(0,1){20}}\put(320,30){\line(0,1){20}}
\put(280,10){\line(0,1){20}}\put(320,-10){\line(0,1){20}}
\put(340,-150){\line(0,1){220}}
\put(-20,-50){\line(0,1){40}}\put(260,-50){\line(0,1){40}}
\put(0,-70){\line(0,1){40}}\put(280,-70){\line(0,1){40}}
\put(40,-50){\line(0,1){20}} \put(320,-50){\line(0,1){20}}
\put(40,-90){\line(0,1){20}}\put(320,-90){\line(0,1){20}}
\put(40,-130){\line(0,1){20}}\put(320,-130){\line(0,1){20}}
\put(0,-110){\line(0,1){20}}\put(280,-110){\line(0,1){20}}
\multiput(20,-160)(0,20){18}{\circle*{3}}
\multiput(300,-160)(0,20){18}{\circle*{3}}
\end{picture}

\end{center}
\begin{center}
{\it Figure 3.}
\end{center}

\textit{The left (resp.) right figure illustrates the links or
sub-chains in $(*),(**)$) (resp. $(*),(***)$) of \ref {3.4}. The
conventions of Figures 1,2 apply. Compared to Figure 2 one has
$\ell=k+2$.}

\bigskip

If $\iota_1$ is reduced to an isolated value, one checks (taking
account of the alternating signs of the $\epsilon_i$) that
$\beta^*_{t_{2\ell-2}}$ is linked to $\beta^*_{t_{2k-1}-1}$ via
\ref {3.4}$(***)$.  Moreover the latter is connected by a
sub-chain to $\beta^*_{t_{2k-2}}$ which is linked to
$\beta^*_{t_{2k-3}}$ in turn linked to $\beta^*_{t_{2k-1}}$.  (See
Figure 3). In this we remark that
$\beta'_{t_{2k-1}}=\beta_{t_{2k-1}}+\iota_2$ and
$\beta'_{t_{2k-2}}=\beta_{t_{2k-2}}+\iota_1$.

As before we may summarize the above as
$$\beta^*_{t_{2j-1}-1}-\beta^*_{t_{2k-3}-1}-\beta^*_{t_{2\ell-2}-1}-
\beta^*_{t_{2k-1}}-\beta^*_{t_{2k-3}}-\beta^*_{t_{2k-2}}-\beta^*_{t_{2k-1}-1}-
\beta^*_{t_{2\ell-2}}-\beta^*_{t_{2\ell-1}-1}. \eqno{(**)}$$

Except for the two extreme terms the links or sub-chains between
these elements are illustrated in the right hand side of Figure
$3$.

In both cases $\beta^*_{t_{2j-1}-1}$ is connected via a sub-chain
to $\beta^*_{t_{2\ell-1}-1}$ and the process is then repeated.
Notice that in both cases the sub-chain passes through all the
$\beta^*$ of the left hand or right hand side of the figure.

We may summarize the above by the
\begin {lemma}  Condition a) of \ref {2.6} is satisfied by $\Pi^*$.
\end{lemma}

\textbf{Remark}.  This result can be read off more easily though
less rigorously from Figure $4$ which is a paradigm for the
general case (except when $\iota_1$ is an isolated value; but then
one combines it with Figure $3$).

\subsection{}\label{3.7.0}

We define a partial order $\leq$ on $I$ as follows.  The smallest
elements are those $i \in I$ for which $\beta_i$ is unchanged.  If
$\iota_1$ as defined in \ref {3.4} is reduced to an isolated value
say $\beta_{j-1}$, then we set $\beta_j > \beta_{j-1}$ recalling
that $\beta_{j-1}$ is left unchanged and
$\beta_j'=\beta_j+\beta_{j-1}$.

If $m=1$ in \ref {3.4}$(***)$ , then $\beta_s$ occurring there is
changed by a simple interval value namely $\iota_3$.  However for
the present purposes it is convenient to view this as a compound
interval value.  Then the largest elements of $i \in I$ are just
those for which $\beta_i$ is changed by adding a compound interval
value, namely when $i=t-1,s$ in the notation of \ref {3.4}, within
that double signature change.   Continuing with this convention we
may easily observe that if $\beta'_j=\beta_j+\iota$ is a simple
interval value, then the $\beta_i$ in the support of $\iota$ are
unchanged (see Figure 2, for example).  Thus if we let $j \in I$
for which $\beta_j$ is changed by a simple interval value, to be
the second largest elements of I, it follows that $\leq$ is
well-defined and lifts to a total order (which we also denote by
$\leq$) giving the

\begin {lemma}  With respect to $\leq$ the transformation taking
the $\epsilon_i\beta_i$ to $\beta_i^*$ is triangular with ones on
the diagonal.
\end {lemma}

\subsection{}\label{3.7}

\begin {prop} $\Pi^*$ satisfies conditions a),
b) and d) of \ref {2.6}.
\end {prop}

\begin {proof}  Condition a) is just Lemma \ref {3.6}.
Condition b) is verified by Lemma \ref {2.8} and the fact that the
added interval values are always a sum of an odd number of
sequentially adjacent simple interval values.

For condition d) we remark that the $\epsilon_i\beta_i: i \in I$
are roots and therefore are either positive or negative roots with
respect to $\Pi^*$.  Yet they must be positive roots by Lemma \ref
{3.7.0}.

\end {proof}

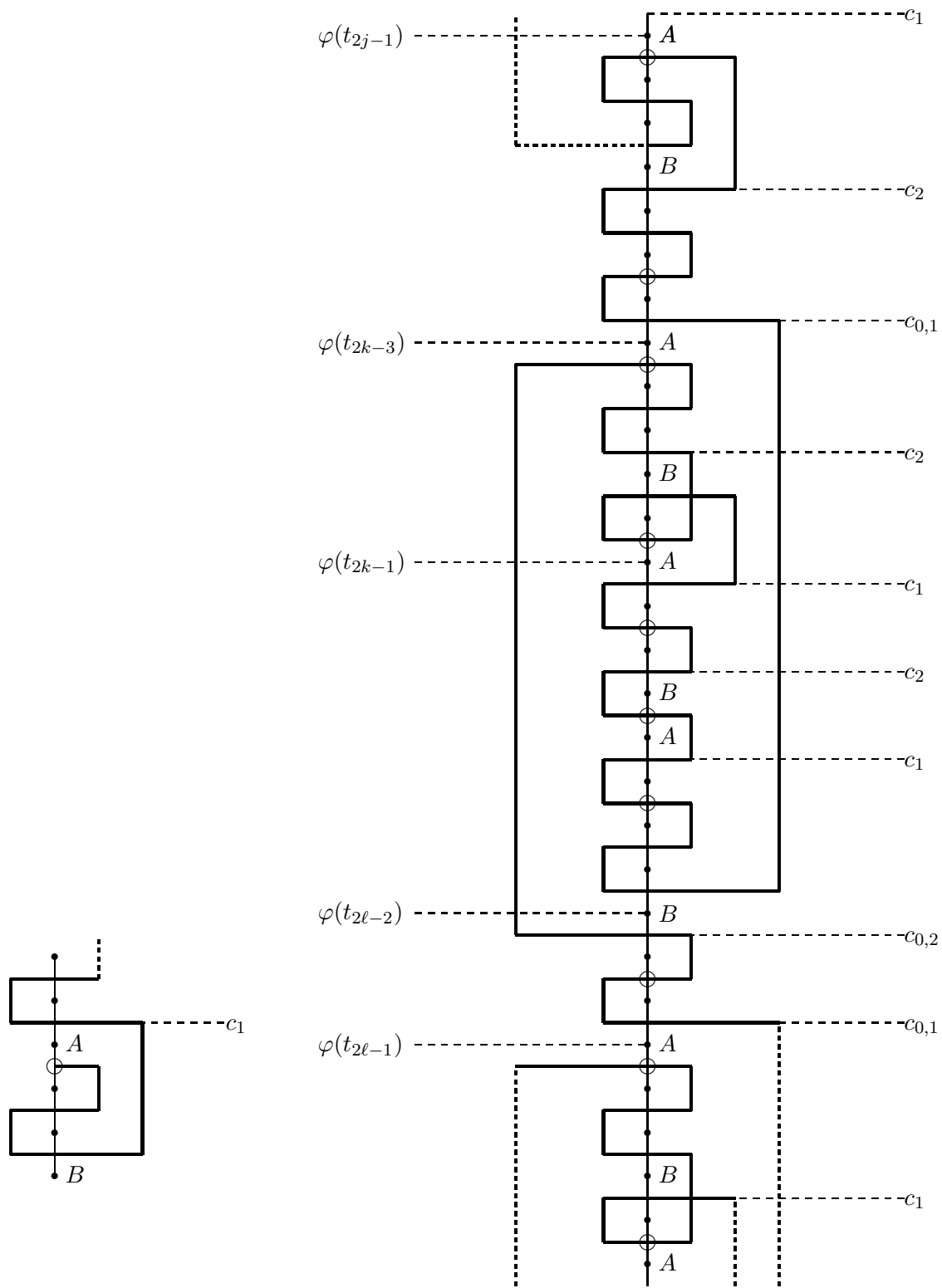
\begin{center}

\begin{picture}(120,600)(-60,-230)
\put(-220,-170){\line(0,1){100}} \put(50,-220){\line(0,1){580}}

\multiput(50,360)(6,0){20}{\line(1,0){3}} \put(167,357){$c_1$}
\multiput(50,280)(6,0){20}{\line(1,0){3}} \put(167,277){$c_2$}
\multiput(50,220)(6,0){20}{\line(1,0){3}} \put(167,217){$c_{0,1}$}
\multiput(50,160)(6,0){20}{\line(1,0){3}}\put(167,157){$c_2$}
\multiput(50,100)(6,0){20}{\line(1,0){3}}\put(167,97){$c_1$}
\multiput(50,60)(6,0){20}{\line(1,0){3}}\put(167,57){$c_2$}
\multiput(50,20)(6,0){20}{\line(1,0){3}}\put(167,17){$c_1$}
\multiput(50,-60)(6,0){20}{\line(1,0){3}}\put(167,-63){$c_{0,2}$}
\multiput(50,-100)(6,0){20}{\line(1,0){3}}\put(167,-103){$c_{0,1}$}
\multiput(-200,-100)(6,0){10}{\line(1,0){3}}\put(-142,-103){$c_1$}
\multiput(50,-180)(6,0){20}{\line(1,0){3}}\put(167,-183){$c_1$}

\multiput(-56,350)(6,0){18}{\line(1,0){3}}\put(-100,347){$\varphi(t_{2j-1})$}
\multiput(-56,210)(6,0){18}{\line(1,0){3}}\put(-100,207){$\varphi(t_{2k-3})$}
\multiput(-56,110)(6,0){18}{\line(1,0){3}}\put(-100,107){$\varphi(t_{2k-1})$}
\multiput(-56,-50)(6,0){18}{\line(1,0){3}}\put(-100,-53){$\varphi(t_{2\ell-2})$}
\multiput(-56,-110)(6,0){18}{\line(1,0){3}}\put(-100,-113){$\varphi(t_{2\ell-1})$}

\put(55,347){$A$}
\put(55,347){$A$}\put(55,207){$A$}\put(55,107){$A$}
\put(55,27){$A$}\put(55,-113){$A$}\put(55,-213){$A$}\put(-215,-113){$A$}

\put(55,287){$B$}\put(55,147){$B$}\put(55,47){$B$}
\put(55,-54){$B$}\put(55,-173){$B$}\put(-215,-173){$B$}

\put(50,340){\circle{7}} \put(50,240){\circle{7}}
\put(50,200){\circle{7}} \put(50,120){\circle{7}}
\put(50,80){\circle{7}} \put(50,40){\circle{7}}
\put(50,0){\circle{7}} \put(50,-80){\circle{7}}
\put(50,-120){\circle{7}}
\put(50,-200){\circle{7}}\put(-220,-120){\circle{7}}

\linethickness{0.4mm} 

\put(50,300){\line(1,0){20}} \put(-220,-120){\line(1,0){20}}
\put(30,320){\line(1,0){40}}\put(30,260){\line(1,0){40}}\put(30,240){\line(1,0){40}}
\put(30,180){\line(1,0){40}}\put(30,160){\line(1,0){40}}\put(30,120){\line(1,0){40}}
\put(30,60){\line(1,0){40}}\put(30,40){\line(1,0){40}}\put(30,20){\line(1,0){40}}
\put(30,80){\line(1,0){40}}\put(30,0){\line(1,0){40}}\put(30,-20){\line(1,0){40}}
\put(30,-80){\line(1,0){40}}\put(30,-140){\line(1,0){40}}\put(30,-160){\line(1,0){40}}
\put(30,-200){\line(1,0){40}}\put(-240,-140){\line(1,0){40}}\put(-240,-80){\line(1,0){40}}

\put(30,340){\line(1,0){60}}\put(30,280){\line(1,0){60}}
\put(30,140){\line(1,0){60}}\put(30,100){\line(1,0){60}}\put(30,-180){\line(1,0){60}}
\put(-240,-160){\line(1,0){60}}\put(-240,-100){\line(1,0){60}}

\put(30,220){\line(1,0){80}}\put(-10,200){\line(1,0){80}}
\put(30,-40){\line(1,0){80}}\put(-10,-60){\line(1,0){80}}
\put(30,-100){\line(1,0){80}}\put(-10,-120){\line(1,0){80}}

\put(30,320){\line(0,1){20}} \put(70,300){\line(0,1){20}}
\put(30,260){\line(0,1){20}} \put(70,240){\line(0,1){20}}
\put(30,220){\line(0,1){20}} \put(70,180){\line(0,1){20}}
\put(30,160){\line(0,1){20}} \put(30,120){\line(0,1){20}}
\put(30,80){\line(0,1){20}} \put(70,60){\line(0,1){20}}
\put(30,40){\line(0,1){20}} \put(70,20){\line(0,1){20}}

\put(30,0){\line(0,1){20}} \put(70,-20){\line(0,1){20}}
\put(30,-40){\line(0,1){20}} \put(70,-80){\line(0,1){20}}
\put(30,-100){\line(0,1){20}} \put(30,-160){\line(0,1){20}}
\put(30,-200){\line(0,1){20}} \put(70,-140){\line(0,1){20}}

\put(-200,-140){\line(0,1){20}}
\put(-240,-100){\line(0,1){20}}\put(-240,-160){\line(0,1){20}}

\put(70,120){\line(0,1){40}}
\put(90,100){\line(0,1){40}}\put(70,-200){\line(0,1){40}}

\put(90,280){\line(0,1){60}} \put(-180,-160){\line(0,1){60}}

\put(-10,-60){\line(0,1){260}}\put(110,-40){\line(0,1){260}}

\multiput(50,-210)(0,20){29}{\circle*{3}}
\multiput(-220,-170)(0,20){6}{\circle*{3}}
\multiput(-10,300)(4,0){15}{\line(1,0){2}}

\multiput(-10,300)(0,4){15}{\line(0,1){2}}
\multiput(-10,-220)(0,4){25}{\line(0,1){2}}
\multiput(110,-220)(0,4){30}{\line(0,1){2}}
\multiput(90,-220)(0,4){10}{\line(0,1){2}}
\multiput(-200,-80)(0,4){5}{\line(0,1){2}}

\end{picture}

\end{center}
\begin{center}
{\it Figure 4.} \\
\end{center}

\textit{A diagrammatic presentation of the proof of Lemma \ref
{3.6}, though of course one could increase the size of the regions
of positive and negative signature. The conventions are those of
Figures 1,2. The diagram on the left describes what happens if the
last element of $B$ is an end point. Again if
$\varphi(t_{2\ell-2})$ is replaced by $\varphi(n)=b$, that is to
say $\ell$ is not defined, all lines which lie partly or
completely below the end point $b$ are removed. If the first
element of $A$ is an end point then one omits the dashed lines
starting just above the first element of $B$. In this case the
lines start at the undecided point, namely $\varphi(t_2)$.}

\subsection{}\label{3.8}

By Lemma \ref {2.10}, condition c) of \ref {2.6} can only fail to
hold if $\beta_e$ is an unchanged boundary value (that is
$\beta_e'=\beta_e$) of a unique turning point $\varphi(t) \in B$
which in particular does not have an isolated value as a
neighbour. This may be either an internal turning point or an end
point. The latter case is essentially a degeneration of the former
which we consider first.

The procedure to obtain $\Pi^*$ must be modified in the above
situation.  To be transparent we first recall some features of the
description of $\Pi^*$ and then define the modified simple root
system which we shall denote by $\Pi^{**}$.  Here we recall that
$\Pi=\{\beta_i\}_{i \in I}$ is a simple root system of type
$A_{n-1}$.  In this $\beta_{i-1}, \beta_{i+1}$ will be said to be
the neighbours of $\beta_i$.

In what follows $\iota_1,\iota_2$ are interval values but
\textit{not} necessarily those defined in \ref {3.4}.

Let $\beta_f$ denote the second boundary value of the internal
turning point $\varphi(t)$ defined above. By our assumption and
rule 2) of \ref {3.1}, this boundary value must be changed, that
is to say we have
$$\beta_f'=\beta_f+\iota_1,$$ where $\iota_1$ is an interval value.
Moreover this interval value starts at $\beta_e$ and so
$\iota_1-\beta_e$ is a root. Since $\beta_f$ is a non-nil boundary
value, Lemma \ref {2.8} forces it to admit a neighbour
$\beta_{f'}$ in the simple interval containing $\beta_f$. Moreover
again by Lemma \ref {2.8} either $\beta_{f'}$ is not a boundary
value or it is nil and so by rules 1), 2) of \ref {3.1}, it is
unchanged.  That is $$\beta'_{f'} =\beta_{f'}.$$ From the
definition of $\iota_1$ given in \ref {3.4}, \ref {3.5}, it
follows that $\beta_{f'}+\beta_f+\iota_1$ is root.

The fact that $$\beta'_e=\beta_e,$$ means that either there is a
unique $i(e) \in I$, or simply $i$, such that
$$\beta'_i=\beta_i + \iota_2,$$ for some interval value $\iota_2$,
with $\beta'_i-\beta_e$ being a root, or $\beta^*_e$ is at the end
point of the Dynkin diagram for $\Pi^*$ and in this case we say
that $i(e)$ is not defined.

Since $\beta_e$ is a non-nil boundary value, Lemma \ref {2.8}
forces it to admit a neighbour $\beta_{e'}$ in the simple interval
containing $\beta_e$. The same argument for $\beta_{f'}$ given
above shows that
$$\beta'_{e'}=\beta_{e'}.$$

From the definition of $\iota_2$ given in \ref {3.4}, \ref {3.5},
it follows that $\iota_2-\beta_e$ and
$(\iota_2-\beta_e)-\beta_{e'}$ are roots.

Since $\beta_i$ is a non-nil boundary value to an element of $A$,
it admits by Lemma \ref {2.8}, a unique neighbour $\beta_{i'}$ in
the same interval. Moreover $\beta_{i'}$ cannot be a nil boundary
value (because then it would be a boundary value to an element of
$B$ and this would contradict Lemma \ref {2.7}(ii)) and it cannot
be a non-nil boundary value either (because this would contradict
Lemma \ref {2.8}).  Hence it is not a boundary value and so is
unchanged by rule 1) of \ref {3.1}, that is
$$\beta_{i'}'=\beta_{i'}.$$  In particular the fact that $\beta_i$
and $\beta_{i'}$ are neighbours implies that $\beta_i'$ and
$\beta_{i'}'$ are neighbours and we designate this as $\beta_i'
\frac{-}{\quad} \beta_{i'}'$, with the sign being that of the
scalar product.  (For the starred quantities $\beta^*_i=\epsilon_i
\beta_i$, one recalls that \textit{all} scalar products have
non-positive signs.)

As in \ref {3.6}, one checks that $\beta_{e'}=\beta_{e'}'$  is
linked \textit{via} $\iota_1$ through a chain defined by
non-vanishing scalar products of neighbours to $\beta_f'=\beta_f
+\iota_1$. We write this as $\beta_{e'}'\frac{-}{\quad}
\ldots\frac{+}{ \quad } \beta_f'$, with the signs having the same
meaning as before. Thus (previous to our proposed modification) we
obtain the chain
$$\beta_{i'}' \frac{-}{\quad} \beta_i'\frac{+}{\quad}
\beta_e'\frac{-}{\quad}\beta_{e'}' \frac{-}{\quad}
\ldots\frac{+}{\quad}
\beta_f'\frac{-}{\quad}\beta_{f'}'.\eqno{(*)}$$ 

Now we make the following modification (using a double prime to
make the distinction clear and writing
$\beta^{**}_i:=\epsilon_i\beta''_i$, with
$\Pi^{**}=\{\beta_i^{**}\}_{i \in I}$). Here exactly three double
primed elements are distinct from the single primed elements and
only these are described below.

Set $$\beta_i''=\beta_i +\iota_2-\beta_e, \beta_e''=-\iota_1,
\beta''_f=\beta_f.$$

Using the above observations and in particular the linking role of
$\iota_1$, we obtain the chain
$$\beta_{i'}''\frac{-}{\quad}  \beta_i''\frac{+}{\quad}
\beta_{e'}'' \frac{-}{\quad} \ldots \frac {-}{\quad}
\beta_e''\frac{+}{\quad} \beta_f''\frac{-}{\quad}
\beta_{f'}''.\eqno{(**)}$$

If $i(e)$ is not defined then the first two terms in both $(*)$
and $(**)$ are absent.

The transition between $(*)$ and $(**)$ above is illustrated in
the passage of the left to the right hand side of Figure $5$
(resp. Figure $6$) when $\iota_1=\iota_2$ is a simple (resp.
compound) interval value. Observe how the long link between
$\beta'_{e'}$ and $\beta'_f$, propagated in the first instance
through $\iota_1$, becomes transformed to a long link between
$\beta''_{e'}$ and $\beta''_e$ similarly propagated in the first
instance through $\iota_1$.

\begin{center}

\begin{picture}(600,200)(-10,50)
\put(20,50){\line(0,1){160}} \put(300,50){\line(0,1){160}}
\multiput(40,190)(5,0){22}{\line(1,0){3}}
\multiput(170,190)(5,0){22}{\line(1,0){3}}

\multiput(40,170)(5,0){22}{\line(1,0){3}}
\multiput(180,170)(5,0){20}{\line(1,0){3}}

\multiput(40,130)(5,0){22}{\line(1,0){3}}
\multiput(170,130)(5,0){22}{\line(1,0){3}}

\multiput(40,110)(5,0){22}{\line(1,0){3}}
\multiput(170,110)(5,0){22}{\line(1,0){3}}

\multiput(40,90)(5,0){22}{\line(1,0){3}}
\multiput(170,90)(5,0){22}{\line(1,0){3}}

\multiput(40,70)(5,0){22}{\line(1,0){3}}
\multiput(170,70)(5,0){22}{\line(1,0){3}}

\multiput(-80,170)(5,0){20}{\line(1,0){3}}
\multiput(320,170)(5,0){16}{\line(1,0){3}}

\multiput(-80,90)(5,0){20}{\line(1,0){3}}
\multiput(320,110)(5,0){16}{\line(1,0){3}}

\put(155,187){$\beta_{i'}$}

\put(155,167){$\beta_{i(e)}$}

\put(155,127){$\beta_{e'}$}

\put(155,107){$\beta_e$}

\put(155,87){$\beta_f$}

\put(155,67){$\beta_{f'}$}

\put(-90,167){$c_1$} \put(-90,87){$c_2$}

\put(400,167){$c'_1$} \put(400,107){$c'_2$}

\put(25,155){$A$} \put(305,155){$A$} 

 \put(25,95){$B$} \put(305,95){$B$}

\put(20,150){\circle{9}} \put(300,150){\circle{9}}

\linethickness{0.4mm}
\put(0,190){\line(1,0){40}}\put(280,190){\line(1,0){40}}
\put(0,130){\line(1,0){40}}\put(0,70){\line(1,0){40}}
\put(280,150){\line(1,0){40}}\put(280,110){\line(1,0){40}}
\put(280,90){\line(1,0){40}}\put(280,70){\line(1,0){40}}

\put(-20,170){\line(1,0){60}} \put(260,170){\line(1,0){60}}
\put(0,150){\line(1,0){60}} \put(260,130){\line(1,0){60}}
\put(-20,110){\line(1,0){60}} \put(0,90){\line(1,0){60}}

\put(40,170){\line(0,1){20}} \put(320,170){\line(0,1){20}}
\put(0,130){\line(0,1){20}} \put(320,130){\line(0,1){20}}
\put(40,110){\line(0,1){20}} \put(320,90){\line(0,1){20}}
\put(0,70){\line(0,1){20}} \put(280,70){\line(0,1){20}}

\put(260,130){\line(0,1){40}} \put(280,110){\line(0,1){40}}

\put(-20,110){\line(0,1){60}} \put(60,90){\line(0,1){60}}

\multiput(20,60)(0,20){8}{\circle*{3}}
\multiput(300,60)(0,20){8}{\circle*{3}}
\end{picture}

\end{center}
\begin{center}
{\it Figure 5.}
\end{center}

\textit{Passing from left to right illustrates the transition from
$(*)$ to $(**)$ in \ref {3.8}, when $\iota_1=\iota_2$ and is a
simple interval value. The conventions of Figure 1 apply.  The new
changed values according to the discussion following $(*)$ of \ref
{3.8} are indicated by a prime.}

\bigskip

\begin{center}

\begin{picture}(600,300)(-10,-40)
\put(20,-30){\line(0,1){240}} \put(300,-30){\line(0,1){240}}
\multiput(40,190)(5,0){22}{\line(1,0){3}}
\multiput(170,190)(5,0){22}{\line(1,0){3}}

\multiput(40,170)(5,0){22}{\line(1,0){3}}
\multiput(180,170)(5,0){20}{\line(1,0){3}}

\multiput(40,50)(5,0){22}{\line(1,0){3}}
\multiput(170,50)(5,0){22}{\line(1,0){3}}

\multiput(40,30)(5,0){22}{\line(1,0){3}}
\multiput(170,30)(5,0){22}{\line(1,0){3}}

\multiput(40,10)(5,0){22}{\line(1,0){3}}
\multiput(170,10)(5,0){22}{\line(1,0){3}}

\multiput(40,-10)(5,0){22}{\line(1,0){3}}
\multiput(170,-10)(5,0){22}{\line(1,0){3}}

\multiput(-80,170)(5,0){20}{\line(1,0){3}}
\multiput(340,170)(5,0){16}{\line(1,0){3}}

\multiput(-80,130)(5,0){20}{\line(1,0){3}}
\multiput(-80,90)(5,0){20}{\line(1,0){3}}

\multiput(320,130)(5,0){20}{\line(1,0){3}}
\multiput(320,90)(5,0){20}{\line(1,0){3}}

\multiput(-80,10)(5,0){20}{\line(1,0){3}}
\multiput(320,30)(5,0){20}{\line(1,0){3}}

\put(155,187){$\beta_{i'}$}

\put(155,167){$\beta_{i(e)}$}

\put(155,47){$\beta_{e'}$}

\put(155,27){$\beta_e$}

\put(155,7){$\beta_f$}

\put(155,-13){$\beta_{f'}$}

\put(-100,167){$c_{0,1}$} \put(-100,7){$c_{0,2}$}

\put(420,167){$c'_{0,1}$} \put(420,27){$c'_{0,2}$}

\put(-90,127){$c_1$} \put(-90,87){$c_1$}

\put(420,127){$c_1$} \put(420,87){$c_1$}

\put(25,155){$A$} \put(305,155){$A$} \put(25,95){$A$}
\put(305,95){$A$}

 \put(25,115){$B$} \put(305,115){$B$}
 \put(25,15){$B$} \put(305,15){$B$}

\put(20,150){\circle{9}} \put(300,150){\circle{9}}
\put(20,110){\circle{9}} \put(300,110){\circle{9}}
\put(20,70){\circle{9}} \put(300,70){\circle{9}}

\linethickness{0.4mm}
\put(0,190){\line(1,0){40}}\put(280,190){\line(1,0){40}}
\put(0,130){\line(1,0){40}}
\put(0,110){\line(1,0){40}}\put(280,190){\line(1,0){40}}
\put(0,90){\line(1,0){40}} \put(0,70){\line(1,0){40}}
\put(0,50){\line(1,0){40}}\put(0,-10){\line(1,0){40}}

\put(0,70){\line(1,0){40}}

\put(280,170){\line(1,0){40}}\put(280,130){\line(1,0){40}}
\put(280,110){\line(1,0){40}}
\put(280,90){\line(1,0){40}}\put(280,70){\line(1,0){40}}
\put(280,10){\line(1,0){40}}\put(280,-10){\line(1,0){40}}

 \put(280,150){\line(1,0){50}} \put(280,30){\line(1,0){50}}

\put(0,170){\line(1,0){60}} \put(280,170){\line(1,0){60}}
\put(-20,150){\line(1,0){60}} \put(-20,10){\line(1,0){60}}
\put(-20,10){\line(1,0){60}} \put(280,50){\line(1,0){60}}
\put(0,30){\line(1,0){60}}

\put(0,170){\line(0,1){20}} \put(280,170){\line(0,1){20}}
\put(40,130){\line(0,1){20}} \put(280,130){\line(0,1){20}}
\put(0,110){\line(0,1){20}} \put(320,110){\line(0,1){20}}
\put(40,90){\line(0,1){20}} \put(280,90){\line(0,1){20}}
\put(0,70){\line(0,1){20}} \put(320,70){\line(0,1){20}}
\put(40,50){\line(0,1){20}} \put(280,50){\line(0,1){20}}
\put(0,30){\line(0,1){20}}
\put(280,10){\line(0,1){20}}\put(40,-10){\line(0,1){20}}
\put(320,-10){\line(0,1){20}}

\put(340,50){\line(0,1){120}} \put(330,30){\line(0,1){120}}

\put(-20,10){\line(0,1){140}} \put(60,30){\line(0,1){140}}

\multiput(20,-20)(0,20){12}{\circle*{3}}
\multiput(300,-20)(0,20){12}{\circle*{3}}
\end{picture}

\end{center}
\begin{center}
{\it Figure $6$.}
\end{center}

\textit{Passing from left to right illustrates the transition from
$(*)$ to $(**)$ in \ref {3.8}, when $\iota_1=\iota_2$ and is a
compound interval value. The conventions of Figure 1 apply. The
new changed values according to the discussion following $(*)$ of
\ref {3.8} are indicated by a prime.}

\bigskip

Finally let us suppose that $\varphi(t)$ is an end point lying
(necessarily) in $B$.  Let $\varphi(s)$ be the closest turning
point to $\varphi(t)$.  By Lemma \ref {2.4} we have $\varphi(s)
\in A$. Denote the interval value $\iota_{t,s}$ simply by $\iota$.

Suppose that the unique nil boundary value to the turning point
$\varphi(s)$ lies in the support of  $I_{t,s}$ (rather than in the
support of the adjacent interval). If $\varphi(s)$ admits a second
boundary value say $\beta_i$, then this is changed to
$\beta_i':=\beta_i+\iota$, noting here that $i=i(e)$.  Moreover we
obtain the chain described in $(*)$ except that the two terms on
the right hand side are not defined. Then parallel to the above we
write
$$\beta_i''=\beta_i +\iota-\beta_e, \beta_e''=-\iota.$$ This gives
the chain described in $(**)$ except that the two terms on the
right hand side are not defined.

If either the unique nil boundary value to $\varphi(s)$ does not
lie in the support of $I_{t,s}$ (that is to say that it lies in
the support of the adjacent interval or $\varphi(s)$ is also an
end point) then $i(e)$ is not defined and we obtain $(*)$ with
just the two central terms. Then we set
$$\beta_e''=-\iota$$ and we obtain $(**)$ with just the two
central terms.

\begin {thm} In the above construction of $\Pi^{**}$, conditions a)-d),
of \ref {2.6} hold.
\end {thm}

\begin {proof}

Retain the notation of \ref {3.8}.

Comparison of $(*)$ and $(**)$ shows that condition a) for $\Pi^*$
(verified in Proposition \ref {3.7}) implies condition a) for
$\Pi^{**}$.

To show that condition d) holds observe that we still have
$\epsilon_i\beta_i=\beta^*_i+\sum_{j\in I}\mathbb N\beta^*_j$, for
all $i \in I\setminus \{e\}$ and so these elements are positive
roots with respect to $\Pi^*$.

By contrast $\beta_e''=-\iota_1=-\beta_e-\ldots,$ and so
$\epsilon_e\beta_e$ does \textit{not} lie in $\mathbb N\Pi^{**}$.
However condition d) does not require this.

Thus condition d) holds for $\Pi^{**}$.

Observe that $\iota_1$ expressed as an element of $\pi$ admits a
positive (resp. negative) coefficient of $\alpha_p$ if $\beta_e$
lies above (resp. below) $\varphi(t) \in B$ and that
$\epsilon_e=1$ (resp. $-1$).  We conclude that condition b) also
holds for $\beta^{**}_e$ and hence for $\Pi^{**}$.

Finally condition c) holds by construction.
\end {proof}

\subsection{}\label{3.9}

The case when positive and negative need to be interchanged in
\ref {3.4}, which can occur if $p$ is even, is similar.  We sketch
the necessary changes below.

Recall the notation of \ref {3.3} and in particular that $j_1=1$.
Before we had assumed the first of the two possibilities in \ref
{3.3}, namely that sg$(j_1)=1$, to hold.  Now we assume the second
possibility, namely that sg$(j_1)=-1$, to hold. Fix a positive odd
integer $u\leq r$ and set $j = j_u, k=j_{u+1}, \ell =j_{u+2}$.

When only $j$ is defined or when all three are defined the
solution we adopt is just the ``mirror image" of that described in
\ref {3.4}.  Indeed reading indices in the opposite direction
(more precisely applying the involution $i\mapsto n-i$ to $I$ it
follows that for example a negative to positive to negative
signature change becomes a positive to negative to positive
signature change. (Here we do \textit{not} mean to imply a
configuration produced by a coprime pair will transform to a
configuration produced by another coprime pair.  Indeed our
formalism allows for some configurations not necessarily coming
from coprime pairs. See \ref {3.10}.)

It remains to consider the case of a negative to positive
signature change, that is when just $j,k$ above are defined.
Surprisingly this is not quite a degeneration of the case when all
three are defined.

We remark that to describe a negative to positive signature change
repeated more than once (say twice to be specific) then we match a
negative to positive to negative signature change with a negative
to positive signature change (as illustrated by the right hand
side of Figure $7$ below).

Set $t=t_{2k-2}$ and let $s$ be the last turning point which
occurs in $A$.  One easily checks that the undecided element $d$
is just $\varphi(t)$.   We leave $\beta_t$ unchanged (which is
just what we would do if $\ell$ were defined) and set
$\beta'_{t-1}=\beta_{t-1}+\iota_{t,s}$.   The result is
illustrated in Figure 7. In view of this last change we have
labelled the changed element $\beta'_{t-1}$ by $c_{0,2}$ in Figure
$7$.

\begin{center}

\begin{picture}(120,460)(-60,-160)

\put(-200,120){\line(0,1){70}} \put(50,-130){\line(0,1){400}}

\multiput(73,180)(6,0){20}{\line(1,0){3}} \put(192,177){$c_2$}
\multiput(73,140)(6,0){20}{\line(1,0){3}} \put(192,137){$c_{0,1}$}
\multiput(73,100)(6,0){20}{\line(1,0){3}} \put(192,97){$c_{0,2}$}
\multiput(73,20)(6,0){20}{\line(1,0){3}} \put(192,17){$c_1$}
\multiput(73,-40)(6,0){20}{\line(1,0){3}} \put(192,-43){$c_2$}
\multiput(73,-80)(6,0){20}{\line(1,0){3}} \put(192,-83){$c_1$}

\multiput(-183,140)(6,0){9}{\line(1,0){3}} \put(-133,137){$c_1$}

\multiput(-195,190)(6,0){18}{\line(1,0){3}}
\multiput(-40,190)(6,0){15}{\line(1,0){3}}
\put(-85,187){$\varphi(t_{2(j-1)})$}

\multiput(-55,90)(6,0){18}{\line(1,0){3}}
\put(-75,87){$\varphi(t)$}

\multiput(-40,10)(6,0){15}{\line(1,0){3}}
\put(-85,7){$\varphi(t_{2k-1})$}

\multiput(-55,-90)(6,0){18}{\line(1,0){3}}
\put(-75,-93){$\varphi(s)$}


\put(-195,147){$A$}\put(55,247){$A$} \put(55,147){$A$}
\put(55,7){$A$} \put(55,-93){$A$}

\put(-195,187){$B$} \put(55,187){$B$} \put(55,87){$B$}
\put(55,-33){$B$} \put(55,-133){$B$}


\put(-200,160){\circle{7}} \put(50,240){\circle{7}}
\put(50,160){\circle{7}} \put(50,120){\circle{7}}
\put(50,40){\circle{7}} \put(50,0){\circle{7}}
\put(50,-60){\circle{7}} \put(50,-100){\circle{7}}

\linethickness{0.4mm}

\put(50,80){\line(1,0){20}} \put(-200,160){\line(1,0){20}}

\put(-220,180){\line(1,0){40}}\put(-220,140){\line(1,0){40}}

\put(30,220){\line(1,0){40}}

\put(30,120){\line(1,0){40}} \put(30,60){\line(1,0){40}}
\put(30,40){\line(1,0){40}}

\put(30,20){\line(1,0){40}} \put(30,-20){\line(1,0){40}}
\put(30,-60){\line(1,0){40}}

\put(30,-80){\line(1,0){40}} \put(30,-120){\line(1,0){40}}

\put(30,240){\line(1,0){60}}\put(30,180){\line(1,0){60}}
\put(10,200){\line(1,0){60}}

\put(30,160){\line(1,0){80}}

\put(10,0){\line(1,0){60}}

\put(10,-40){\line(1,0){60}}

\put(-10,140){\line(1,0){80}}

\put(30,100){\line(1,0){80}} \put(30,-100){\line(1,0){80}}

%
%

\put(-180,160){\line(0,1){20}}\put(70,200){\line(0,1){20}}
\put(30,220){\line(0,1){20}}\put(30,160){\line(0,1){20}}
\put(70,120){\line(0,1){20}} \put(30,100){\line(0,1){20}}
\put(70,60){\line(0,1){20}}

\put(30,40){\line(0,1){20}} \put(70,20){\line(0,1){20}}
\put(70,-20){\line(0,1){20}}

\put(70,-60){\line(0,1){20}} \put(30,-80){\line(0,1){20}}
\put(30,-120){\line(0,1){20}}

\put(-220,140){\line(0,1){40}}\put(30,-20){\line(0,1){40}}

\put(10,-40){\line(0,1){40}} \put(70,-120){\line(0,1){40}}

\put(90,180){\line(0,1){60}}

\put(110,-100){\line(0,1){200}}

%
\multiput(-200,130)(0,20){4}{\circle*{3}}
\multiput(50,-130)(0,20){20}{\circle*{3}}

\multiput(-180,120)(0,4){5}{\line(0,1){2}}
\multiput(-10,140)(0,4){32}{\line(0,1){2}}
\multiput(10,200)(0,4){15}{\line(0,1){2}}
\multiput(110,160)(0,4){27}{\line(0,1){2}}
\multiput(10,260)(4,0){15}{\line(1,0){2}}
\multiput(70,260)(0,4){2}{\line(0,1){2}}

\end{picture}

\end{center}
\begin{center}
{\it Figure 7.} \\
\end{center}

\textit{This illustrates the negative to positive signature change
as discussed and in the notation of \ref {3.9}, more precisely in
the special case when $j=k-1$. The conventions are those of
Figures 1,2. The left hand side describes the situation when
$t_{2(j-1)}$ is a starting point. The right hand side describes
the situation when $\varphi(t_{2k-1})$ has positive signature but
is preceded by a component of negative signature and this being
repeated any number of times. Here the top half of the diagram
must and does match the negative to positive to negative
configuration which can be obtained by inverting Figure 4.
Curiously the bottom half of the diagram does not match the
positive to negative to positive configuration from Figure 4, nor
indeed need it do so. An end point of $\Pi^*$ in the sense of its
Dynkin diagram is $\beta^*_t$. For the diagram on the left a
second end point is $\beta^*_{t_{2j-1}-1}$.}

\bigskip

The modification of this construction when the exceptional value
is not changed is exactly as in \ref {3.8}. The extension of
Theorem \ref {3.8} to this case is proved similarly.

\subsection{}\label{3.10}

Although the solution we gave to Conditions a)-d) of \ref {2.6} is
unambiguous, one can easily check that other solutions can exist
for certain coprime pairs $p,q$.

The number of coprime pairs grows at most linearly with $n$ whilst
the number of possible signatures grows exponentially. Thus to
obtain all possible signatures one would need to have increasingly
large gaps between turning points.  In particular we do not claim
that the particular arrangements described in Figure 4 actually
arise from a coprime pair.

\subsection{}\label{3.11}

Combining Theorem \ref {3.8} with the remarks in \ref {1.8}, \ref
{2.6} and \ref {3.9}, we obtain the

\begin {cor}  Suggestion \ref {1.4} holds for a (truncated)
biparabolic of index $1$.
\end {cor}

\end{document}